\documentclass[12pt,reqno]{amsart}

\usepackage{amsmath, amsthm, amsopn, amssymb, enumerate}

\usepackage{latexsym }
\usepackage{graphics}
\usepackage{graphicx} 
\usepackage{verbatim} 
\usepackage{hyperref} 
\usepackage[mathscr]{eucal}

\newtheorem{thm}{Theorem}[section]
\newtheorem{lemma}[thm]{Lemma}
\newtheorem{cor}[thm]{Corollary}
\newtheorem{conj}[thm]{Conjecture}
\newtheorem{prop}[thm]{Proposition}

\newtheorem*{claim*}{Claim}

\theoremstyle{definition}

\newtheorem{defn}[thm]{Definition}

\newtheorem{rmk}[thm]{Remark}
\newtheorem*{Alg}{Algorithm for finding cycles}

\newcommand{\ds}{\displaystyle}

\def\A{\mathcal{A}}
\def\B{\mathcal{B}}
\def\C{\mathcal{C}}
\def\D{\mathcal{D}}

\def\F{\mathcal{F}}
\def\G{\mathcal{G}}
\def\HH{\mathcal{H}}
\def\I{\mathcal{I}}
\def\J{\mathcal{J}}
\def\K{\mathcal{K}}

\def\M{\mathcal{M}}

\def\P{\mathcal{P}}
\def\Q{\mathcal{Q}}
\def\R{\mathcal{R}}
\def\S{\mathcal{S}}

\def\N{\mathbb{N}}

\def\RR{\mathbb{R}}

\def\a{\mathbf{a}}

\def\k{\mathbf{k}}

\def\le{\leqslant}
\def\ge{\geqslant}

\def\eps{\varepsilon}
\def\<{\langle}
\def\>{\rangle}

\def\ex{\textup{ex}}

\begin{document}

\author{Robert Morris} 

\author{David Saxton}

 \address{
   Robert Morris, David Saxton \hfill\break
    IMPA, Estrada Dona Castorina 110, Jardim Bot\^anico, Rio de Janeiro, RJ, Brasil
 }
 \email{\{rob|saxton\}@impa.br}

\thanks{Research supported in part by a CNPq bolsa PDJ (DS) and by CNPq Proc.~479032/2012-2 and Proc.~303275/2013-8 (RM)}

\title{The number of $C_{2\ell}$-free graphs}

\begin{abstract}
One of the most basic questions one can ask about a graph $H$ is: how many $H$-free graphs on $n$ vertices are there? For non-bipartite $H$, the answer to this question has been well-understood since 1986, when Erd\H{o}s, Frankl and R\"odl proved that there are $2^{(1 + o(1)) \ex(n,H)}$ such graphs. For bipartite graphs, however, much less is known: even the weaker bound $2^{O(\ex(n,H))}$ has been proven in only a few special cases: for cycles of length four and six, and for some complete bipartite graphs. 

For even cycles, Bondy and Simonovits proved in the 1970s that $\ex(n,C_{2\ell}) = O\big( n^{1 + 1/\ell} \big)$, and this bound is conjectured to be sharp up to the implicit constant. In this paper we prove that the number of $C_{2\ell}$-free graphs on $n$ vertices is at most $2^{O(n^{1 + 1/\ell})}$, confirming a conjecture of Erd\H{o}s. Our proof uses the hypergraph container method, which was developed recently (and independently) by Balogh, Morris and Samotij, and by Saxton and Thomason, together with a new `balanced supersaturation theorem' for even cycles. We moreover show that there are at least $2^{(1 + c)\ex(n,C_6)}$ $C_6$-free graphs with $n$ vertices for some $c > 0$ and infinitely many values of~$n \in \N$, disproving a well-known and natural conjecture. As a further application of our method, we essentially resolve the so-called Tur\'an problem on the Erd\H{o}s-R\'enyi random graph $G(n,p)$ for both even cycles and complete bipartite graphs. 
\end{abstract}

\maketitle

\section{Introduction}

One of the central challenges in graph theory is to determine the extremal and typical properties of the family of $H$-free graphs on $n$ vertices. For non-bipartite graphs, an enormous amount of progress has been made on this problem; for bipartite graphs, on the other hand, surprisingly little is known. For example, the extremal number $\ex(n,H)$ (the maximum number of edges in an $H$-free graph on $n$ vertices) was determined asymptotically for all non-bipartite $H$ over 60 years ago, but, despite much effort, even its order of magnitude is known for only a handful of bipartite graphs. A significantly harder question asks: how many $H$-free graphs are there with $n$ vertices? In particular, Erd\H{o}s asked more than thirty years ago (see, e.g.,~\cite{KW82}) whether or not the number of such graphs is at most $2^{O(\ex(n,H))}$ for every bipartite graph $H$, but the answer is known in only a few special cases, see~\cite{BSmm,BSst,KW96,KW82}. 

In this paper we prove that the number of $C_{2\ell}$-free graphs is at most $2^{O(n^{1 + 1/\ell})}$, confirming a longstanding conjecture of Erd\H{o}s. Our method is very general, and is likely to apply to various other classes of bipartite graphs; in particular, we show that a similar bound holds for any bipartite graph which has a certain `balanced supersaturation' property. We also essentially resolve the Tur\'an problem on the Erd\H{o}s-R\'enyi random graph $G(n,p)$ for both even cycles and complete bipartite graphs, obtaining close to best possible bounds for all values of $p$. Finally, we show that the natural conjecture (often attributed to Erd\H{o}s) that the number of $H$-free graphs on $n$ vertices is $2^{(1 + o(1))\ex(n,H)}$ fails for $H = C_6$. 

\subsection{History and background} 

The study of extremal graph theory was initiated roughly 70 years ago by Tur\'an~\cite{T41}, who determined exactly the extremal number of the complete graph, by Erd\H{o}s and Stone~\cite{ES46}, who determined asymptotically (for all $r \ge 3$) the extremal number of a complete $r$-partite graph\footnote{As was first pointed out by Erd\H{o}s and Simonovits~\cite{ES66}, this is sufficient to determine $\ex(n,H)$ asymptotically for every non-bipartite graph $H$.}, and by K\H{o}v\'ari, S\'os and Tur\'an~\cite{KST} who showed that $\ex(n,K_{s,t}) = O(n^{2 - 1/s})$, where $K_{s,t}$ denotes the complete bipartite graph with part sizes~$s$ and~$t$. (The case $K_{2,2} = C_4$ was solved some years earlier by Erd\H{o}s~\cite{E38} during his study of multiplicative Sidon sets.) Over the following decades, a huge amount of effort was put into determining more precise bounds for specific families of graphs (see, e.g.,~\cite{MGT,FuS}), and a great deal of progress has been made. Nevertheless, the order of magnitude of $\ex(n,H)$ for most bipartite graphs, including simple examples such as $K_{4,4}$ and $C_8$, remains unknown.

\enlargethispage{\baselineskip}

In the 1970s, the problem of determining the number of $H$-free graphs on $n$ vertices was introduced by Erd\H{o}s, Kleitman and Rothschild~\cite{EKR}, who proved that there are $2^{(1 + o(1))\ex(n,K_r)}$ $K_r$-free graphs, and moreover that almost all triangle-free graphs are bipartite. This latter result was extended to all cliques by Kolaitis, Pr\"omel and Rothschild~\cite{KPR} and to more general graphs by Pr\"omel and Steger~\cite{PS}, and the former to all non-bipartite graphs by Erd\H{o}s, Frankl and R\"odl~\cite{EFR}, using Szemer\'edi's regularity lemma. The corresponding result for $k$-uniform hypergraphs was proved by Nagle, R\"odl and Schacht~\cite{NRS} using hypergraph regularity, and reproved by Balogh, Morris and Samotij~\cite{BMS} and Saxton and Thomason~\cite{ST} using the hypergraph container method (see below). Much more precise results for graphs were obtained by Balogh, Bollob\'as and Simonovits~\cite{BBS1,BBS2,BBS3}.

For bipartite $H$ the problem seems to be significantly harder, and much less is known. The first important breakthrough was made by Kleitman and Winston~\cite{KW82} in 1982, who showed that there are at most $2^{(1+c)\ex(n,C_4)}$ $C_4$-free graphs on $n$ vertices, where $c \approx 1.17$, improving the trivial upper bound of $n^{\ex(n,C_4)}$, and getting within striking distance of the trivial lower bound $2^{\ex(n,C_4)}$. Their result moreover resolved a longstanding open question posed by Erd\H{o}s (see~\cite{KW82}). However, it was not until almost 30 years later that their theorem was extended to other complete bipartite graphs, by Balogh and Samotij~\cite{BSmm,BSst}, who proved, for every $2 \le s \le t$, that there are at most $2^{O(n^{2 - 1/s})}$ $K_{s,t}$-free graphs on $n$ vertices. Their bound is conjectured to be sharp up to the constant implicit in the $O(\cdot)$, but constructions giving a matching lower bound are known only when either $s \in \{2,3\}$ or $t > (s-1)!$, see~\cite{ARS,Brown,ERS,F96,KRS}.

For other (i.e., non-complete) bipartite graphs, the only known bounds of this form are for forests, where the problem is much easier, and for even cycles of length six and eight. Recall that $\ex(n,C_{2\ell}) = O(n^{1 + 1/\ell})$ for every $\ell \ge 2$.\footnote{The first published proof of this bound was given by Bondy and Simonovits~\cite{BS74}, but they attribute the result to Erd\H{o}s, see also~\cite{E64} and~\cite[Theorem~4.6]{FuS}. For more recent improvements, see~\cite{BZ,P12,V00}.} Erd\H{o}s and Simonovits conjectured (see~\cite{E64} or~\cite{FuS}) that this bound is sharp up to the implied constant factor, but matching lower bounds are known only for $C_4 = K_{2,2}$ (see above), $C_6$ and $C_{10}$ (see~\cite{Benson,FNV,LUW,W91}). It was therefore natural for Erd\H{o}s to conjecture that, for every $\ell \ge 2$, the number of $C_{2\ell}$-free graphs is at most $2^{O(n^{1+1/\ell})}$, and indeed Kleitman and Wilson~\cite{KW96} proved this in the cases $\ell = 3$ and $\ell = 4$, using a clever colouring argument to reduce the problem to that solved in~\cite{KW82}. They (and independently Kreuter~\cite{Kreuter}, see also~\cite{KKS}) moreover proved that there are $2^{O(n^{1+1/\ell})}$ graphs with no even cycles of length \emph{at most} $2\ell$. However, they were unable to resolve the case of a single forbidden long even cycle, and no further progress has been made in the decade and a half since.

\subsection{Main results}

In this paper we resolve this longstanding open problem for all even cycles, using a very general method, which we expect to give similar bounds for many other bipartite graphs. More precisely, we shall prove the following theorem.

\begin{thm} \label{thm:main}
For every $\ell \ge 2$, there are at most $2^{O(n^{1+1/\ell})}$ $C_{2\ell}$-free graphs on $n$ vertices.
\end{thm}

As noted above, it is generally believed that the bound in Theorem~\ref{thm:main} is sharp up to the constant implicit in the $O(\cdot)$, but this is only known in the cases $\ell \in \{ 2, 3, 5 \}$. Theorem~\ref{thm:main} is an immediate consequence of the following result, which gives a rough structural description of $C_{2\ell}$-free graphs. 

\begin{thm}\label{thm:cycle:containers}
Given $\ell \ge 2$ and $\delta > 0$, there exists a constant $C = C(\delta,\ell)$ such that the following holds for every sufficiently large $n \in \N$.  There exists a collection~$\G$ of at most
$$2^{\delta n^{1 + 1/\ell}}$$
graphs on vertex set~$[n]$ such that
$$e(G) \le C n^{1+1/\ell}$$ 
for every $G \in \G$, and every $C_{2\ell}$-free graph is a subgraph of some~$G \in \G$. 
\end{thm}

We remark that we shall in fact prove a substantial generalization of Theorem~\ref{thm:cycle:containers}, which will provide us with close to optimal family of  `containers' of any given size, see Theorem~\ref{thm:cycle:containers:turan}.

A closely related structural question asks: how many edges does a typical $H$-free graph on $n$ vertices have? Balogh, Bollob\'as and Simonovits~\cite{BBS3} conjectured that there exists a constant $c > 0$ such that, if $H$ contains a cycle, then almost every $H$-free graph on $n$ vertices has between $c \cdot \ex(n,H)$ and $(1 - c) \ex(n,H)$ edges. This is only known for some complete bipartite graphs~\cite{BSC4,BSst,KW82} and for $C_6$~\cite{KW96} (as usual, much more is known for non-bipartite graphs). The following bound is an immediate consequence of Theorem~\ref{thm:cycle:containers}.

\begin{thm}\label{cor:fewwithfewedges}
The  number of $C_{2\ell}$-free graphs on $n$ vertices with $o\big( n^{1 + 1/\ell} \big)$ edges is $2^{o(n^{1 + 1/\ell})}$.
\end{thm}

\enlargethispage{\baselineskip}

Under the additional assumption that $\ex(n,C_{2\ell}) = \Omega( n^{1 + 1/\ell} )$, this implies that almost all $C_{2\ell}$-free graphs have $\Omega( n^{1 + 1/\ell} )$ edges. Theorem~\ref{cor:fewwithfewedges} may therefore be seen as evidence in favour of the conjecture of Balogh, Bollob\'as and Simonovits.

Another very strong conjecture, often attributed to Erd\H{o}s (see,~e.g.,~\cite{BBS3}), and mentioned explicitly\footnote{More precisely, they said that it ``seems likely" that this holds for every bipartite graph $H$.} by Erd\H{o}s, Frankl and R\"odl~\cite{EFR}, states that the number of $H$-free graphs on $n$ vertices is $2^{(1 + o(1))\ex(n,H)}$ for every graph $H$ that contains a cycle. Recall that it was proved in~\cite{EFR} that this holds for all non-bipartite $H$.

\pagebreak

We shall prove the following proposition, which disproves this conjecture for $C_6$. 

\begin{prop}\label{prop:counterexample}
There exists a constant $c > 0$ such that there are at least
$$2^{(1 + c) \ex(n,C_6)}$$
$C_6$-free graphs on $n$ vertices for infinitely many values of $n \in \N$. 
\end{prop}

We will prove Proposition~\ref{prop:counterexample} using bounds on $\ex(n,C_6)$ due to F\"uredi, Naor and Verstra\"ete~\cite{FNV}. However, we will also prove a similar result for various forbidden families of short cycles, using only the Erd\H{o}s-Bondy-Simonovits bound on the extremal number $\ex(n,C_{2\ell})$. For example, we will give an extremely simple proof that if $\F = \{K_3,C_6\}$, then there are at least $2^{(1+c) \ex(n,\F)}$ $\F$-free graphs on $n$ vertices for infinitely many values of $n$. We conjecture that a similar result holds for all even cycles of length at least six, but our method fails (though not by much!) for $C_4$, and so we are not sure whether or not to expect the conjecture to hold for this and other complete bipartite graphs. It is not inconceivable that the number of $H$-free graphs on $n$ vertices is $2^{(1 + o(1))\ex(n,H)}$ for every $H$ such that $\ex(n,H) = \Omega(n^{3/2})$. 

\subsection{Balanced supersaturation for even cycles, and hypergraph containers}

\enlargethispage{\baselineskip}
\enlargethispage{\baselineskip}

We shall prove Theorem~\ref{thm:cycle:containers} using the hypergraph container method, which was introduced recently by Balogh, Morris and Samotij~\cite{BMS} and by Saxton and Thomason~\cite{ST}. This technique, which allows one to find a relatively small family of sets (`containers') which cover the independent sets in an $r$-uniform hypergraph, has already found many applications; for example, it implies deterministic analogues of the recent breakthrough results of Conlon and Gowers~\cite{CG} and Schacht~\cite{Schacht} on extremal results in sparse random sets, and in many cases proves stronger `counting' versions of those theorems. In order to apply this method, we need to bound, for every graph $G$ with $\gg \ex(n,H)$ edges\footnote{In this paper we shall use the notation $a \gg b$ informally to mean that $a/b$ is bounded from below by a sufficiently large constant.}, a particular parameter (see Section~\ref{sec:containers}) of the hypergraph which encodes copies of our forbidden graph $H$ in $G$. Bounding this parameter in the case $H = C_{2\ell}$ is the main technical challenge of this paper.

The following `balanced supersaturation theorem for even cycles' is our second main result. The existence of a collection satisfying part~$(a)$ was proved\footnote{The proof of this `supersaturation result for even cycles' was not published at the time, but will appear shortly in a paper of Faudree and Simonovits~\cite{FS}.} by Simonovits (see~\cite{ES84}), and was conjectured by Erd\H{o}s and Simonovits~\cite{ES84} to exist for every bipartite graph $H$. The condition in part~$(b)$ is new, and is crucial to our application of the container method. 

\begin{thm}\label{thm:cycle:hypergraph}
For every $\ell \ge 2$, there exist constants $C > 0$, $\delta > 0$ and $k_0 \in \N$ such that the following holds for every $k \ge k_0$ and every $n \in \N$. Given a graph~$G$ with~$n$ vertices and $k n^{1+1/\ell}$ edges, there exists a collection~$\HH$ of copies of $C_{2\ell}$ in $G$, satisfying:
\begin{itemize}
 \item[$(a)$] $|\HH|  \ge \delta k^{2\ell} n^2$, and\smallskip
 \item[$(b)$] $d_\HH(\sigma) \le C \cdot k^{2\ell - |\sigma| - \frac{|\sigma| - 1}{\ell-1}} n^{1 - 1/\ell}$ for every $\sigma \subset E(G)$ with $1 \le |\sigma| \le 2 \ell-1$,\smallskip
\end{itemize}
where $d_\HH(\sigma) = | \{ A \in \HH \,:\, \sigma \subset A \} |$ denotes the `degree' of the set $\sigma$ in $\HH$.
\end{thm}

In words, the theorem above states that if $e(G) \gg n^{1 + 1/\ell}$, then $G$ contains at least as many copies of $C_{2\ell}$ (up to a constant factor) as a (typical) random graph of the same density, and moreover these copies of $C_{2\ell}$ are relatively `uniformly distributed' over the edges of $G$. We emphasize that $\HH$ does not need to include all copies of $C_{2\ell}$ in $G$, but only a subset.

We make the following conjecture for a general bipartite graph $H$, which follows from Theorem~\ref{thm:cycle:hypergraph} in the case $H = C_{2\ell}$. As noted above, the existence in $G$ of as many copies of $H$ as the Erd\H{o}s-R\'enyi random graph with the same number of edges was conjectured by Erd\H{o}s and Simonovits~\cite{ES84}. Since we ask that these copies of $H$ are moreover reasonably uniformly distributed\footnote{The precise form of~\eqref{eq:conjES}, which quantifies this rough description, is slightly artificial: it is essentially the weakest bound compatible with our proof of Proposition~\ref{prop:conj:implies:thm}, see Section~\ref{sec:proof}. We remark that in many cases one would expect something stronger to be true, cf. the bounds proved in Section~\ref{sec:Kst} in the case $H = K_{s,t}$.}, we shall refer to it as the `balanced Erd\H{o}s-Simonovits conjecture'. 

\begin{conj}[Balanced Erd\H{o}s-Simonovits conjecture for general bipartite $H$]\label{conj:balancedES}
Given a bipartite graph $H$, there exist constants $C > 0$, $\eps > 0$ and $k_0 \in \N$ such that the following holds. Let $k \ge k_0$, and suppose that~$G$ is a graph on~$n$ vertices with $k \cdot \ex(n,H)$ edges. Then there exists a (non-empty) collection~$\HH$ of copies of $H$ in $G$, satisfying
\begin{equation}\label{eq:conjES}
d_\HH(\sigma) \le \ds\frac{C \cdot |\HH|}{k^{(1 + \eps)(|\sigma| - 1)} e(G)} \quad \text{for every $\sigma \subset E(G)$ with $1 \le |\sigma| \le e(H)$.}
\end{equation}
\end{conj}

Our motivation in making this conjecture is the following proposition, see also~\cite{Sax}.

\begin{prop}\label{prop:conj:implies:thm}
Let $H$ be a bipartite graph. If Conjecture~\ref{conj:balancedES} holds for $H$, then there are at most $2^{O(\ex(n,H))}$ $H$-free graphs on $n$ vertices. 
\end{prop}

In fact we shall actually prove a slightly more general result (see Section~\ref{sec:proof}), which doesn't require a lower bound on the extremal number of $H$. We remark that, although we do not demand a lower bound on the number of edges of the hypergraph $\HH$ in Conjecture~\ref{conj:balancedES}, we expect that it can be chosen to have (up to a constant factor) as many copies of $H$ as the random graph $G(n,m)$, where $m = k \cdot \ex(n,H)$. We remark that the conjecture holds for the complete bipartite graph $H = K_{s,t}$, under the additional assumption that $\ex(n,K_{s,t}) = \Omega( n ^{2 - 1/s})$. We refer the reader to Section~\ref{sec:Kst} for the precise statement.

\subsection{The Tur\'an problem for random graphs}

\enlargethispage{\baselineskip}

Another consequence of Theorem~\ref{thm:cycle:hypergraph} (which also follows via the hypergraph container method) relates to the so-called `Tur\'an problem on $G(n,p)$', that is, the problem of determining (the typical value of) the maximum number of edges in an $H$-free subgraph of $G(n,p)$. Let us write
$$\ex \big( G(n,p), H \big) \, := \, \max\Big\{ e(G) \,:\, G \subset G(n,p) \textup{ and $G$ is $H$-free} \Big\}$$
and note that both $G(n,p)$ and $\ex \big( G(n,p), H \big)$ are random variables.\footnote{Throughout the paper we will abuse notation slightly by writing $G(n,p)$ to denote both a random variable and the realisation of that random variable.} 

This question has received an enormous amount of attention in recent years (see the excellent survey~\cite{RS}, and the recent breakthroughs in~\cite{BMS,CG,ST,Schacht}). In the case $H = C_{2\ell}$ it was solved over fifteen years ago by Haxell, Kohayakawa and \L uczak~\cite{HKL}, in the following sense: they proved that if $p \gg n^{-1 + 1/(2\ell-1)}$ then $\ex \big( G(n,p), C_{2\ell} \big) \ll e\big( G(n,p) \big)$, whereas if $p = o\big( n^{-1 + 1/(2\ell-1)} \big)$ then $\ex \big( G(n,p), C_{2\ell} \big) = \big( 1 + o(1) \big) e\big( G(n,p) \big)$. Much more precise bounds for a certain range of $p$ were obtained by Kohayakawa, Kreuter and Steger~\cite{KKS}, who showed that, with high probability, 
\begin{equation}\label{eq:KKSthm}
\ex \big( G(n,p), C_{2\ell} \big) \, = \, \Theta\Big( n^{1 + 1/(2\ell-1)} (\log \alpha)^{1/(2\ell-1)} \Big)
\end{equation}
if $p = \alpha n^{-1 + 1/(2\ell-1)}$ and $2 \le \alpha \le n^{1/(2\ell-1)^2}$. However, no non-trivial upper bounds appear to have been obtained for much larger values of $p$.  

Using the hypergraph container method, together with Theorem~\ref{thm:cycle:hypergraph}, we will prove the following upper bounds on $\ex \big( G(n,p), C_{2\ell} \big)$. Both are simple consequences of a structural result (i.e., generalization of Theorem~\ref{thm:cycle:containers}) which we will state and prove in Section~\ref{sec:Turan}. Moreover, by~\eqref{eq:KKSthm} the first bound is sharp up to a polylog-factor, and we will show that, modulo a well-known (and widely believed) conjecture of Erd\H{o}s and Simonovits, the second bound is sharp\footnote{This would also imply that the bound in part~$(b)$ of Theorem~\ref{thm:cycle:hypergraph} is essentially best possible.} up to the value of the constant $C$.

\begin{thm}\label{thm:randomturan}
For every $\ell \ge 2$, there exists a constant $C  = C(\ell) > 0$ such that 
$$\ex \big( G(n,p), C_{2\ell} \big) \, \le \,  \left\{
\begin{array} {c@{\quad}l} 
C n^{1 + 1/(2\ell-1)} (\log n)^2 & \textup{if } \; p \le n^{-(\ell - 1) / (2\ell-1)} (\log n)^{2\ell} \\[+1ex]
C p^{1/\ell} n^{1+1/\ell} & \textup{otherwise} 
\end{array}\right.$$
with high probability as $n \to \infty$. 
\end{thm}

In Section~\ref{sec:Kst} we shall prove a similar (and also probably close to best possible) theorem for $K_{s,t}$-free graphs. We remark that similar results in the closely related setting of $B_h$-sets (though using somewhat different techniques) were also obtained recently in a series of papers  by Dellamonica, Kohayakawa, Lee, R\"odl and Samotij~\cite{DKLRS1,DKLRS2,DKLRS3,KLRS}.

The rest of the paper is organised as follows. First, in Section~\ref{sec:outline}, we give an outline of the proofs of Theorems~\ref{thm:main} and~\ref{thm:cycle:hypergraph} and prove Proposition~\ref{prop:counterexample}. Next, in Section~\ref{Sec:ES}, the most substantial part of the paper, we prove Theorem~\ref{thm:cycle:hypergraph}. In Section~\ref{sec:containers} we formally introduce the hypergraph container method, and in Section~\ref{sec:proof} we will use it to prove Theorems~\ref{thm:main} and~\ref{thm:cycle:containers}, Theorem~\ref{cor:fewwithfewedges} and Proposition~\ref{prop:conj:implies:thm}. We will also give a relatively simple proof of a slightly weaker version of Theorem~\ref{thm:randomturan} (with an extra $\log$-factor), and then, in Section~\ref{sec:Turan}, a more involved proof of the precise statement. Finally, in Section~\ref{sec:Kst}, we will sketch the proof of some similar results with `even cycle' replaced by `complete bipartite graph'.

\section{Preliminaries}\label{sec:outline}

In this section we will prepare the reader for the proofs of the main theorems, and prove the lower bounds claimed in the Introduction. First, in Sections~\ref{sec:containers:outline} and~\ref{sec:superset:sketch}, we will describe the hypergraph container method, and give a sketch of the proof of Theorem~\ref{thm:cycle:hypergraph}. Next, in Section~\ref{sec:lowerbounds}, we will prove Proposition~\ref{prop:counterexample}, as well as similar lower bounds for other families of cycles, and a (conditional) matching lower bound for Theorem~\ref{thm:randomturan}. Lastly, in Sections~\ref{sec:definitions} and~\ref{sec:notation}, we will introduce some of the basic concepts that will be used in the proof of Theorem~\ref{thm:cycle:hypergraph}.

\subsection{The hypergraph container method}\label{sec:containers:outline}

\enlargethispage{\baselineskip}

One of the main results of~\cite{BMS,ST} (see Theorem~\ref{thm:coveroff}, below) states that, given any $r$-uniform hypergraph $\HH$, there exists a relatively small collection of vertex sets (containers), that cover the independent sets of $\HH$, and each of which contains fewer than $(1 - \delta) e(\HH)$ edges of $\HH$. The number of containers depends on the density and `uniformity' of the hypergraph; more precisely, the stronger our upper bounds on the degrees of sets in $\HH$ (as a proportion of the number of edges in $\HH$), the smaller the family of containers is guaranteed to be. We will repeatedly apply this result to the hypergraph produced by Theorem~\ref{thm:cycle:hypergraph}, which encodes (a highly uniform sub-family of) the copies of $C_{2\ell}$ in a graph $G$. The container theorem produces a family of subgraphs of $G$ that forms a cover of the $C_{2\ell}$-free subgraphs of $G$; we then apply the container theorem to each of these graphs, and so on. 

By this method, we obtain a rooted tree $T$ of subgraphs of $K_n$, such that every $C_{2\ell}$-free graph is contained in some leaf of $T$, and each leaf has $O(n^{1 + 1/\ell})$ edges. (To be slightly more explicit, each vertex of this tree corresponds to a graph, the root is $K_n$, and the out-neighbours of each vertex are given by the container theorem described above.) To guarantee that each leaf has $O(n^{1 + 1/\ell})$ edges, we simply apply the container theorem sufficiently many times, noting that Theorem~\ref{thm:cycle:hypergraph} is valid as long as $G$ has more than this many edges. 

It remains to count the leaves of $T$; in order to do so, we need to bound the number of containers formed in each application of Theorem~\ref{thm:coveroff}. This is controlled by a parameter~$\tau$ which (roughly speaking) measures the uniformity of $\HH$, and it turns out that (to deduce Theorem~\ref{thm:main}, for example) we need to be able to apply the theorem with $\tau \approx k^{-(1+\eps)}$, for some $\eps > 0$, when $e(G) = k n^{1 + 1/\ell}$. In order to do so, we shall use properties~$(a)$ and~$(b)$ of Theorem~\ref{thm:cycle:hypergraph}, in particular the fact that our upper bound on $d_\HH(\sigma)$ improves by a factor of more than $k^{1+\eps}$ each time $|\sigma|$ increases by one. In order to deduce Theorem~\ref{thm:randomturan}, on the other hand, we will need the full strength of  Theorem~\ref{thm:cycle:hypergraph}, see Sections~\ref{sec:proof} and~\ref{sec:Turan}. 

\enlargethispage{\baselineskip}

\subsection{The proof of Theorem~\ref{thm:cycle:hypergraph}}\label{sec:superset:sketch}

The most technical part of this paper is the proof of Theorem~\ref{thm:cycle:hypergraph}, in Section~\ref{Sec:ES}. Here we will attempt to give an outline of the key ideas in the proof, and thereby hopefully make it easier for the reader to follow the details of the calculation.

The basic idea, motivated by the proof of Bondy and Simonovits~\cite{BS74}, is as follows: we will find a vertex $x \in V(G)$ and a $t \in \{2,\ldots,\ell\}$ such that the $t^{th}$ neighbourhood $A_t$ of $x$ is no larger than it `should' be. By the conditions $e(G) = k n^{1 + 1/\ell}$ and $k \ge k_0$, such a pair $(x,t)$ must exist (see Lemma~\ref{lem:concentrated:exists}); we will choose a pair with $t$ as small as possible. One can find many cycles in $G$ formed by two paths from $x$ to $A_t$, plus a path of length $2\ell - 2t$ which alternates between the sets $A_{t-1}$ and $A_t$. Repeating this process for a positive proportion of the vertices of $G$, we find at least $\delta k^{2\ell} n^2$ copies of $C_{2\ell}$ in $G$. 

The proof of part~$(a)$, outlined above, is already highly non-trivial, and the requirement that no set of edges of $G$ be contained in too many members of $\HH$ (i.e., copies of $C_{2\ell}$) introduces significant extra complications. We overcome these by substantially modifying our strategy. First, we shall find the cycles in $\HH$ one at a time, selecting carefully from the available choices, instead of simply taking every cycle in $G$ which passes through the vertex $x$. In order to do so, we shall construct (in each step of the process) a sufficiently large sub-family $\C$ of the cycles through $x$ with the following property: no cycle in $\C$ contains any (`saturated') set of edges that are already contained in the maximum allowed number of members of $\HH$. Since $\C$ is sufficiently large, we will be able to deduce that not all of these cycles are already  in $\HH$, and so we can add one of them to our collection.

In order to construct $\C$, we will need to introduce two further types of neighbourhood, which we term the \emph{balanced} and \emph{refined $t$-neighbourhoods} of a vertex $x \in V(G)$. These both consist of a collection of sets $\A = (A_1,\ldots,A_t)$ and a family of paths $\P$ from $x$ to $A_t$, whose $j$th edge ends in $A_j$, which satisfy several further `uniformity' conditions. In particular, for a pair $(\A,\P)$ to be balanced we will require there to be `not too many' sub-paths of $\P$ between any two vertices of $G$, and for a pair to be refined we will require that every vertex of $A_t$ receives `many' paths from $x$. Using the minimality of $t$ (in the $t$-neighbourhood of $x$ chosen above) we will show (see Lemma~\ref{prop:RNF:exists}) that $x$ has a balanced neighbourhood with almost as many paths as one would expect, which avoids all saturated sets of edges, and (see Lemma~\ref{lem:balanced_to_refined}) that every balanced $t$-neighbourhood $(\A,\P)$ contains a refined $t$-neighbourhood~$(\B,\Q)$. 

We now perform the following algorithm. Let $(\A,\P)$ be a balanced $t$-neighbourhood of the vertex $x$, and use the lemma mentioned above to find a refined $t$-neighbourhood $(\B,\Q)$. We form cycles by choosing a path from $x$ to $B_t$, a zig-zag path of length $2\ell - 2t$ between $B_t$ and $B_{t-1}$, and then a path in $\Q$ back to $x$. We repeat this process sufficiently many times, adding to $\C$ only those cycles which avoid all saturated sets of edges. This part of the proof is surprisingly intricate; in particular, one of the key difficulties will be in ensuring we (typically) have many `legal' choices for the path back to $x$. Once we have shown that the family $\C$ thus constructed is sufficiently large, we simply note that each of the cycles passes through one of the edges between $x$ and $B_1$. Assuming that the hypergraph $\HH$ constructed so far does not already have sufficiently many edges, it follows by the pigeonhole principle that one of the cycles of $\C$ is not already a member of $\HH$, as required.

\subsection{Lower bounds, and a proof of Proposition~\ref{prop:counterexample}}\label{sec:lowerbounds}

In this section we will describe an extremely simple method of producing many $\F$-free graphs on $n$ vertices, for certain forbidden families $\F$ consisting of cycles. The proof is based on that of a similar result of Saxton and Thomason~\cite{ST} for Sidon sets. As well as Proposition~\ref{prop:counterexample}, we shall prove the following result, which generalizes the bound stated earlier for the family $\F = \{K_3,C_6\}$.

\begin{prop}\label{prop:general:counter}
There exists a constant $c > 0$ such that the following holds. Let $\ell \ge 3$, and set $\F = \{C_3,\ldots,C_\ell\} \cup \{C_{2\ell}\}$. 
There are at least
$$2^{(1 + c) \ex(n,\F)}$$
$\F$-free graphs on $n$ vertices for infinitely many values of $n \in \N$. 
\end{prop}

\begin{proof}
Choose $c > 0$ sufficiently small\footnote{Since $3^{-4/3} \log_2(34) > 1.17 > (21/20) \cdot 3^{1/20}$, it follows that taking $c = 1/20$ suffices.} so that $(1 + c) 3^{4/3+c} < \log_2 34$, and let $n \in \N$ be such that $\ex(3n,\F) \le 3^{4/3+c} \ex(n,\F)$. Note that since $\ex(n,\F) \le \ex(n,C_{2\ell}) = O(n^{4/3})$, there exist infinitely many such values of $n$. Let $G$ be an extremal graph for $\F$ on $n$ vertices, and let $\G$ be the collection of graphs obtained from $G$ by blowing up each vertex to size three, and replacing each edge by a (not necessarily perfect) matching. Since there are 34 matchings in $K_{3,3}$, it follows that 
$$|\G| \, = \, 34^{\ex(n,\F)} \, > \, 2^{(1 + c) 3^{4/3+c} \ex(n,\F)} \, \ge \, 2^{(1 + c)\ex(3n,\F)}$$
by our choice of $c > 0$ and $n \in \N$. It therefore suffices to show that $\G$ is a family of $\F$-free graphs on $3n$ vertices. To do so, simply note that a cycle of length $k$ in $G' \in \G$ corresponds to a non-backtracking walk in $G$ of length $k$, which must either be a cycle, or contain a cycle of length at most $\lfloor k/2 \rfloor$. This proves that $\G$ is $\F$-free, as required.
\end{proof}

It is easy to see that the proof above can be applied (for example) to any family $\F$ consisting of a graph $H$ with $\ex(n,H) = O(n^{4/3})$ together with all graphs obtained from $H$ by identifying groups of vertices that are pairwise at distance at least 3 from one another. On the other hand, in order to prove Proposition~\ref{prop:counterexample} we will need to apply the following theorem of  F\"uredi, Naor and Verstra\"ete~\cite{FNV}.

\begin{thm}[F\"uredi, Naor and Verstra\"ete, 2005]\label{thm:FNV}
For all sufficiently large $n \in \N$,
$$\ex(n,C_6) < 0.6272 \cdot n^{4/3},$$
and for infinitely many values of $n$ there exists a $\{K_3,C_6\}$-free graph $G$ on $n$ vertices with 
$$e(G) > 0.5338 \cdot n^{4/3}.$$
\end{thm}

\enlargethispage{\baselineskip}

Since $\frac{0.6272}{0.5338} \approx 1.17497 < 1.1758$, we may apply the same argument as in the proof above. We remark that it seems to be an incredible coincidence that the bounds obtained in~\cite{FNV} are almost exactly what we need in order to deduce the proposition.

\begin{proof}[Proof of Proposition~\ref{prop:counterexample}]
Let $G$ be the $\{K_3,C_6\}$-free graph constructed in~\cite{FNV}, with $n/3$ vertices and more than $0.5338 \cdot (n/3)^{4/3}$ edges, where $n \in \N$ is chosen arbitrarily among those integers for which this graph exists. We remark that although in~\cite{FNV} it is not proven that $G$ is triangle-free, this follows easily from their construction via a little case analysis. Now, letting $\G$ be the collection of graphs obtained by blowing up each vertex to size three, and replacing each edge by a matching, it follows that $\G$ is a $C_6$-free family, exactly as in the proof above. Moreover, by Theorem~\ref{thm:FNV}, we have\footnote{Note that $2^{\frac{0.6272}{0.5338} \cdot 3^{4/3}} \approx 33.91$, so we need to take $c < 0.0007$.} 
$$|\G| \, > \, 34^{0.5338 (n/3)^{4/3}} \, > \, 2^{(1 + c) 0.6272 n^{4/3}} \, > \, 2^{(1 + c)\ex(n,C_6)},$$
for some $c > 0$, as required.
\end{proof} 

Since it follows from a similar construction, let us also take this opportunity to show that the bound in Theorem~\ref{thm:randomturan}, and hence also that in Theorem~\ref{thm:cycle:hypergraph}, is essentially best possible, conditional on the following conjecture of Erd\H{o}s and Simonovits~\cite{ES82}.

\begin{conj}[Erd\H{o}s and Simonovits, 1982]\label{conj:ESextremal}
$$\ex\big( n, \{ C_3,C_4,\ldots,C_{2\ell} \} \big) \, = \, \Theta\big( n^{1 + 1/\ell} \big).$$ 
\end{conj}

If the conjecture is true, then there exists (for infinitely many values of $N$) a graph $G$ with~$N$ vertices and $\eps N^{1+1/\ell}$ edges, and girth at least $2\ell + 1$. Choose $p \in (0,1)$ arbitrarily small, set $a = \eps / p$, and blow up each vertex of $G$ into a set of $a$ vertices. Set $n = aN$, and place a copy of the Erd\H{o}s-R\'enyi random graph $G(n,p)$ on these $n$ vertices; that is, choose edges independently at random with probability $p$. We discard all edges inside classes, and between pairs of classes corresponding to non-edges of $G$. For each pair of classes corresponding to an edge of $G$, we retain an arbitrary maximal matching. 

To see that this graph has, with high probability, at least $\eps^3 p^{1/\ell} n^{1 + 1/\ell}$ edges, simply observe that for each edge of $G$, we expect to obtain at least $p a^2 / 2 = \eps^2 / 2p$ edges. (Indeed, we can search for an edge incident to each vertex in turn, ignoring those which have already been used.) Thus the expected number of edges in our $C_{2\ell}$-free subgraph of $G(n,p)$ is at least
$$\frac{\eps^2}{2p} \cdot \eps N^{1+1/\ell} \, = \, \frac{\eps^2}{2p} \cdot \eps \bigg( \frac{p n}{\eps} \bigg)^{1+1/\ell} \, \ge \, \eps^2 \cdot p^{1/\ell} n^{1+1/\ell}.$$
Since the events are independent for different edges of $G$, a standard concentration argument shows that the claimed bound holds with high probability.

\subsection{Saturated sets in good hypergraphs}\label{sec:definitions}

In order to slightly simplify the presentation of the proof of Theorem~\ref{thm:cycle:hypergraph} in Section~\ref{Sec:ES}, we shall prepare the ground in this section by giving a couple of key definitions, and proving a simple lemma. 

Let us fix $\ell \ge 2$ throughout the rest of the paper. We will need various constants in the proof below; we define them here for convenience. Informally, they will satisfy
$$0 \, < \, \delta  \, \ll \, \eps(1) \, \ll \, \eps(2) \, \ll \, \cdots \, \ll \, \eps(\ell) \, \ll \, 1 \, \ll \, C  \, \ll \, k_0.$$
More precisely, we can set $C = 10\ell$, $\eps(\ell) = 1/C^3$, $\eps(t-1) = \eps(t)^{t}$ for each $2 \le t \le \ell$, $\delta = \eps(1)^{3\ell}$ and $k_0 = 1/\delta$. We emphasize that this value of $C$, which is fixed throughout the proof of Theorem~\ref{thm:cycle:hypergraph}, is not the same as the (various different) constants $C$ which appear in the statements in the Introduction. 

For each $n,k \in \N$ and $j \in [2\ell-1]$, set
\begin{equation}\label{def:Delta}
\Delta^{(j)}(k,n) \, = \, \frac{k^{2\ell-1} n^{1 - 1  / \ell}}{ \big( \delta k^{\ell / (\ell-1)} \big)^{j-1}},
\end{equation}
and note that $\Delta^{(j)}(k,n) \ge 1$ if $k \le n^{(\ell-1)/\ell}$. This will represent the maximum degree of a set of~$j$ vertices in the $2\ell$-uniform hypergraph we will construct, whose edges will represent (a subset of the) copies of~$C_{2\ell}$ in a given graph $G$ with~$n$ vertices and~$kn^{1+1/\ell}$ edges. 

\begin{defn}[Good hypergraphs]
We will say that a hypergraph $\HH$ is \emph{good} with respect to $(\delta, k,\ell,n)$ (or simply good), if $d_\HH(\sigma) \le \Delta^{(|\sigma|)}(k,n)$ for every $\sigma \subset V(\HH)$ with $1 \le |\sigma| < 2\ell$.

Similarly, we will say that a collection $\HH$ of copies of $C_{2\ell}$ in a graph $G$ is \emph{good} if the corresponding hypergraph with vertex set $E(G)$ and edge set $\HH$ is good.
\end{defn}

Next, we define~$\F(\HH)$ to be the collection of subsets of~$V(\HH)$ that are at their maximum degree.

\begin{defn}[Saturated sets]
Given a $2\ell$-uniform hypergraph $\HH$ and a set $\sigma \subset V(\HH)$ with $1 \le |\sigma| < 2 \ell$, we say that~$\sigma$ is \emph{saturated} if $d_\HH(\sigma) \ge \lfloor \Delta^{(|\sigma|)}(k,n) \rfloor$. Set
 \[
  \F(\HH) = \big\{ \sigma \subset V(\HH) \,:\, \sigma \mbox{ is saturated} \big\}.
 \]
\end{defn}

We will sometimes need to ensure that a copy of $C_{2\ell}$ contains no saturated set of edges. In avoiding these sets, the following concept will be useful. For each $S \subset V(\HH)$ and $j \in \N$, define the \emph{$j$-link} of~$S$ in~$\F = \F(\HH)$ to be\footnote{We shall write $X^{(j)}$ and $X^{(\le \, j)}$ to denote the collection of subsets of $X$ of size (at most)~$j$. }
\[
 L_\F^{(j)}(S) \,=\, \Big\{ \sigma \in \big( V(\HH) \setminus S \big)^{(j)} \,:\, \sigma \cup \tau \in \F \mbox{ for some non-empty } \tau \subset S \Big\},
\]
and set $L_\F(S) = \bigcup_{j \ge 1} L_\F^{(j)}(S)$. In order to give the reader some practice with the various notions just introduced, let us prove the following easy (and useful) bound on $|L_\F^{(j)}(S)|$.

\begin{lemma}\label{lem:size_of_link}
Let $\HH$ be a good $2\ell$-uniform hypergraph and let~$\F = \F(\HH)$. Then
 \[
  \big| L_\F^{(j)}(S) \big| \, \le \, 2^{2\ell+|S|+1} \cdot \big( \delta k^{\ell/(\ell-1)} \big)^j
 \]
for every $j \in \N$ and every $S \subset V(\HH)$.
\end{lemma}

\begin{proof}
For each non-empty set $\tau \subset S$, set
\[
 \J(\tau) \, = \, \Big\{ \sigma \in \big( V(\HH) \setminus S \big)^{(j)} \,:\, \sigma \cup \tau \in \F \Big\}.
\]
Now, by the handshaking lemma and the definition of goodness, we have 
$$2^{-2\ell} \sum_{\sigma \in \J(\tau)} d_\HH(\sigma \cup \tau) \, \le \, d_\HH(\tau) \, \le \,  \Delta^{(|\tau|)}(k,n),$$
since each edge of~$\HH$ is counted at most $2^{2\ell}$ times in the sum (once for each~$\sigma \in \J(\tau)$ which is a subset of that edge). Moreover
$$\sum_{\sigma \in \J(\tau)} d_\HH(\sigma \cup \tau) \, \ge \, |\J(\tau)| \cdot \lfloor \Delta^{(|\tau|+j)}(k,n) \rfloor,$$
by the definitions of~$\J$ and~$\F$. Thus
\begin{equation}\label{eq:Deltaratio}
 |\J(\tau)| \, \le \, 2^{2\ell} \cdot \frac{\Delta^{(|\tau|)}(k,n)}{\lfloor \Delta^{(|\tau|+j)}(k,n) \rfloor} \, \le \, 2^{2\ell+1} \cdot \big( \delta k^{\ell/(\ell-1)} \big)^j.
\end{equation}
Finally, since the sets $\J(\tau)$ cover $L_\F^{(j)}(S)$, it follows that
\[
 \big| L_\F^{(j)}(S) \big| \, \le \sum_{\emptyset \,\neq\, \tau \,\subset\, S} |\J(\tau)| \, \le \, 2^{2\ell+|S|+1} \cdot  \big( \delta k^{\ell/(\ell-1)} \big)^j,
\]
as required.
\end{proof}

\enlargethispage{\baselineskip}

\subsection{Notation}\label{sec:notation}

Let us finish this section by introducing a few more pieces of notation which will prove useful in the proof of Theorem~\ref{thm:cycle:hypergraph}. First, a \emph{$t$-neighbourhood} of a vertex $x \in V(G)$ is simply a collection $\A = (A_1,\ldots,A_t)$ of (not necessarily disjoint) sets of vertices such that $A_i \subset N(A_{i-1})$ for each $i \in [t]$ (here, and throughout, we set $A_0 = \{x\}$). In the following definitions, let $\P$ be a collection of paths $(x,u_1,\ldots,u_t)$ of length $t$ in $G$, where $u_i \in A_i$. 

We denote by 
$$\P_{i,j} \, = \, \big\{ (u_i,\ldots,u_j) \,:\, (x, u_1,\ldots, u_t) \in \P \big\}$$ 
the set of paths which travel between the $i$th and the $j$th vertices\footnote{By convention, we give the initial vertex of a path label zero, so $u_i$ is the $i$th vertex of $(x, u_1,\ldots, u_t)$.} of a path in~$\P$. Moreover, given $u,v \in V(G)$ and a collection of paths $\Q$ (for example, $\Q = \P_{i,j}$), we write
$$\Q[u \to v] \, = \, \big\{ (x_1,\ldots,x_s) \in \Q \,:\, x_1 = u \text{ and } x_s = v \big\}$$ 
for the set of paths in $\Q$ which begin at $u$ and end at $v$. Similarly, for $S \subset V(G)$ we write $\Q[u \to S] = \bigcup_{v \in S} \Q[u\to v]$ for the set of paths in~$\Q$ that start at~$u$ and end in~$S$.

Given a family of edge-sets $\F$, we will say that a path $P \in \P$ \emph{avoids} $\F$ if $E(P)$ contains no member of $\F$ as a subset, and similarly that $\P$ avoids $\F$ if every $P \in \P$ avoids $\F$. 

Given a vertex $v \in V(G)$ and $r \in \{0,\ldots,t-1\}$, the \emph{$r$-branching factor} of $v$ in $\P$ is the maximum $d$ such that there exist $d$ paths in $\P$ with the $r$th vertex $v$, and pairwise distinct $(r+1)$st vertices. The branching factor of $\P$ is the maximum branching factor of a vertex in $\P$. Finally, if $\A = (A_1,\ldots,A_t)$ and $\B = (B_1,\ldots,B_t)$ are sequences of sets of vertices of the same length, then we write $\A \prec \B$ to denote the fact that $A_i \subset B_i$ for every $i \in [t]$.

\section{The balanced Erd\texorpdfstring{\H{o}}{o}s-Simonovits conjecture for even cycles}\label{Sec:ES}

In this section we shall prove Theorem~\ref{thm:cycle:hypergraph}, using the following key proposition, which allows us to build up the good hypergraph~$\HH$ representing cycles in~$G$ one cycle at a time.
Recall that an integer $\ell \ge 2$, and constants $\delta > 0$ and $k_0 \in \N$ were fixed above.

\begin{prop}\label{prop:finding:cycles}
Let $n,k \in \N$, with $k \ge k_0$, let $G$ be a graph with $n$ vertices and $kn^{1+1/\ell}$ edges, and let $\HH$ be a collection of copies of $C_{2\ell}$ in $G$ that is good with respect to $(\delta,k,\ell,n)$. If $e(\HH) \le \delta k^{2\ell} n^2$, then there exists a copy $H \not\in \HH$ of $C_{2\ell}$ such that $\HH \cup \{H\}$ is good.
\end{prop}

We remark that Theorem~\ref{thm:cycle:hypergraph} follows easily by repeatedly applying Proposition~\ref{prop:finding:cycles}, see Section~\ref{ESconj:proofSec} for the details.

\subsection{A sketch of the proof}

The proof Proposition~\ref{prop:finding:cycles} is quite long and technical, and so we shall begin with a brief outline of the proof, see also Section~\ref{sec:superset:sketch}. We shall introduce three different types of `$t$-neighbourhood' of a vertex $x \in V(G)$, which we term \emph{concentrated}, \emph{balanced} and \emph{refined}, respectively, each of which is more restrictive than the previous one. 

The first step in the proof is to choose a vertex $x \in V(G)$ with a concentrated $t$-neighbourhood (see Definition~\ref{defn:concentrated_nbd}), with $t$ as small as possible. Next, we show that $x$ moreover has a balanced $t$-neighbourhood $(\A,\P)$ (see Definition~\ref{def:balanced:nbhd}) containing roughly as many paths as one would expect, and avoiding $\F = \F(\HH)$ (see Lemma~\ref{prop:RNF:exists}). We then show that $x$ has a refined $t$-neighbourhood $(\B,\Q)$ (see Definition~\ref{def:refined:nbhd}) with $\B \prec \A$ and $\Q \subset \P$ (see Lemma~\ref{lem:balanced_to_refined}). The minimality of $t$ is crucial in the proofs of Lemmas~\ref{prop:RNF:exists} and~\ref{lem:balanced_to_refined}.

Having completed these preliminaries, we then construct a collection $\C$ of copies of $C_{2\ell}$ by choosing a path from $x$ to $B_t$, a zigzag path of length $2\ell - 2t$ between $B_t$ and $B_{t-1}$, and a path in $\Q$ back to $x$. We use the properties of a refined $t$-neighbourhood to show that there are many such cycles, and to control the number of cycles that contain a saturated set of edges. Finally, we use our assumption that $\HH$ is good, together with the pigeonhole principle, to show that the collection $\C$ is not contained in $\HH$, as required.

\subsection{Balanced \texorpdfstring{$t$}{t}-neighbourhoods}\label{sec:balanced}

In this section we will lay the groundwork for the proof of Proposition~\ref{prop:finding:cycles} by finding a vertex in $V(G)$ with a particularly `well-behaved' collection of paths of (some well-chosen) length $t$ leaving it. To avoid repetition, let us fix integers $n,k \in \N$ with $k \ge k_0$, a graph $G$ with at most $n$ vertices and at least $k n^{1 + 1/\ell}$ edges, and a collection $\HH$ of at most $\delta k^{2\ell} n^2$ copies of $C_{2\ell}$ in $G$ which is good with respect to $(\delta,k,\ell,n)$. We will also write $\HH$ for the corresponding $2\ell$-uniform hypergraph $\HH$ with vertex set $E(G)$. 

To begin, observe that, since $e(\HH) \le \delta k^{2\ell} n^2$, there are at most 
$$\frac{2\ell \cdot e(\HH)}{\Delta^{(1)}(k,n)} \, \ll \, e(G)$$ 
saturated edges of $G$. Thus, by choosing a (non-empty) subgraph of~$G$ if necessary (and weakening the bound on $e(G)$ slightly), we may assume that
\begin{equation}
 \label{eq:Delta1_bounded}
 d_\HH(e) \,<\, \Delta^{(1)}(k,n) \quad\mbox{ for every $e\in E(G)$,} 
\end{equation}
i.e., that none of the edges of $G$ are saturated. In fact, since $\delta = \eps(1)^{3\ell}$, we may (and will) assume the stronger bound, that 
\begin{equation}\label{eq:deg1bound}
 d_\HH(e) \,<\, \eps(1)^{3\ell-1} k^{2\ell-1} n^{1-1/\ell} \quad\mbox{ for every $e\in E(G)$.} 
\end{equation}
Similarly, since there are at most $C \eps(\ell) k n^{1 + 1/\ell} \ll e(G)$ edges of $G$ incident to vertices of degree at most $C \eps(\ell) k n^{1/\ell}$, we may also assume that~$\delta(G) \ge C \eps(\ell) k n^{1/\ell}$. 

The first step is to define, for each vertex $x \in V(G)$, a distance $t(x) \in [2,\ell]$ at which the neighbourhood of $x$ becomes sufficiently concentrated.

\begin{defn}\label{defn:concentrated_nbd}
Let $x \in V(G)$, and let $\A = ( A_1, \ldots, A_t )$ be a collection of (not necessarily disjoint) subsets of $V(G)$. We say that $\A$ is a {\em concentrated $t$-neighbourhood of~$x$} if 
\begin{equation}\label{eq:concentrated}
|A_t| \le k^{(\ell-t)/(\ell-1)} n^{t/\ell},
\end{equation}
but $|N(v) \cap A_i| \ge C\eps(t) k n^{1/\ell}$ for every $v \in A_{i-1}$ and every $i \in [t]$, where $A_0 = \{x\}$. 

Define $t(x)$ to be the minimal $t \ge 2$ such that there exists a concentrated $t$-neighbourhood of~$x$ in $G$ (set $t(x) = \infty$ if none exists), and define $t(G) = \min\{ t(x) : x \in V(G) \}$.
\end{defn}

The following easy observation uses the fact that $\delta(G) \ge C \eps(\ell) k n^{1/\ell}$.

\begin{lemma}\label{lem:concentrated:exists}
$t(x) \le \ell$ for every $x \in V(G)$.
\end{lemma}

\begin{proof}
The proof is an easy consequence of the definition: simply set $A_i = N(A_{i-1})$ for each $1 \le i \le \ell$, and note that the condition~\eqref{eq:concentrated} holds trivially for $t = \ell$. 
\end{proof}

As noted above, the cycles in $\C$ will all pass through some vertex $x$ with $t(x) = t(G)$. The next step is to show that, given such a vertex $x$, there exists a `well-behaved' collection of paths in $G$, each starting at $x$ and of length $t(G)$. The following definition gathers the properties that we will require this collection to possess. 

\begin{defn}[Balanced $t$-neighbourhoods]\label{def:balanced:nbhd}
Let $x \in V(G)$ and $2 \le t \in \N$. Set $A_0 = \{x\}$, and suppose that 
\begin{itemize}
\item $\A = ( A_1, \ldots, A_t )$ is a collection of (not necessarily disjoint) sets of vertices of $G$, \smallskip
\item $\P$ is a collection of paths in~$G$ of the form $(x,u_1, \ldots, u_t)$, with $u_i \in A_i$ for each $i \in [t]$. 
\end{itemize}
We will say that the pair $(\A,\P)$ forms a {\em balanced $t$-neighbourhood of $x$} if the following conditions hold:
\begin{enumerate} 
\item[$(i)$]
The first and last levels are not too large, that is: 
\begin{equation}
|A_1| \, \le \, kn^{1/\ell} \qquad \text{and} \qquad |A_t| \, \le \, k^{(\ell-t)/(\ell-1)}n^{t/\ell}.
\end{equation}
 \item[$(ii)$]
 For every $i,j$ with $0 \le i < j \le t$ and $(i,j) \ne (0,t)$, and every pair $u\in A_i$, $v \in A_j$, 
\begin{equation}\label{eq:def:prop5}
|\P_{i,j}[u \to v]| \, \le \, k^{(j-i-1)\ell/(\ell-1)}.
\end{equation}
\item[$(iii)$]
The branching factor of~$\P$ is at most $C\eps(t)kn^{1/\ell}$. 
\end{enumerate}
\end{defn}

Note that in a balanced $t$-neighbourhood $\P$ could be empty, so we will always additionally require a lower bound (of order $(k n^{1/\ell})^t$) on the number of paths; however, the exact value of this lower bound will vary depending on the situation. The purpose of condition~$(i)$ is to allow us to apply the pigeonhole principle later on; conditions~$(ii)$ and~$(iii)$ are designed to prevent too many paths from being removed when we `refine' the neighbourhood in Section~\ref{sec:finding:cycles}.  

The main objective of this subsection is to prove the following lemma. Recall that the graph $G$ and the hypergraph $\HH$ were fixed earlier, and set $t := t(G)$.

\enlargethispage{\baselineskip}

\begin{lemma}\label{prop:RNF:exists}
Let $x \in V(G)$. If $t(x) = t$, then $G$ contains a balanced $t$-neighbourhood $(\A,\P)$ of $x$, with
\begin{equation}\label{eq:number:of:paths}
 |\P| \, \ge \, \frac{C^t}{2} \big( \eps(t) kn^{1/\ell} \big)^t,
 \end{equation}
such that $\P$ avoids $\F = \F(\HH)$.
\end{lemma}

In order to find a collection of paths as in Lemma~\ref{prop:RNF:exists}, we simply take a large collection of paths of length $t$ from $x$ (which is guaranteed to exist by Definition~\ref{defn:concentrated_nbd}), and then remove paths until condition~$(ii)$ holds. The key point is that, if too many paths are removed in this second step, then we will be able to show that $t(x) < t$, which is a contradiction. The following easy lemma is the key tool in this deduction.

\begin{lemma}\label{lem:paths_to_concentrated_nbd}
Let $\R$ be a collection of paths of length $s \ge 2$ in $G$ from a vertex $x \in V(G)$ to a set $B \subset V(G)$. If $|B| \le k^{(\ell-s)/(\ell-1)}n^{s/\ell}$,
$$|\R| \, \ge \, 2Cs \cdot \eps(s) \big( kn^{1/\ell} \big)^s$$
and $\R$ has branching factor at most $kn^{1/\ell}$, then $t(x) \le s$. 
\end{lemma}

\begin{proof}
For each $0 \le i \le s$, set $A_i$ equal to the collection of $i$th vertices of the paths in a collection $\R' \subset \R$ obtained as follows: simply remove (repeatedly) from each $A_i$ (with $0 \le i < s$), any vertex whose $i$-branching factor (in the remaining paths) is less than $C \eps(s) k n^{1/\ell}$, and remove from $\R$ all paths with this as their $i$th vertex. If $\R'$ is non-empty, then it follows by construction that $\A = (A_1, \ldots, A_s)$ is a concentrated $s$-neighbourhood of~$u$, and thus $t(x) \le s$, as claimed.

It therefore suffices to show that not all paths are destroyed in the process described above. To see this, simply note that the number of paths destroyed is at most
$$Cs \cdot \eps(s) \big( k n^{1/\ell} \big)^s \, \le \, |\R| / 2,$$ 
since each destroyed path passes through a vertex of branching factor at most $C \eps(s) k n^{1/\ell}$, and the branching factor of every other vertex is at most $k n^{1/\ell}$. 
\end{proof}

We are now ready to prove Lemma~\ref{prop:RNF:exists}.

\begin{proof}[Proof of Lemma~\ref{prop:RNF:exists}]
Let $\A$ be a concentrated $t$-neighbourhood of $x$, where $t(x) = t(G) = t$. We may and will assume that $|A_1| \le kn^{1/\ell}$, since if~$A_1$ is larger than this, then we can remove vertices from~$A_1$ without destroying the concentrated neighbourhood property. Furthermore since $\A$ is a concentrated $t$-neighbourhood, we have that $|A_t| \le k^{(\ell-t)/(\ell-1)}n^{t/\ell}$.

For every $i \in [t]$ and every $v \in A_{i-1}$, where $A_0 = \{x\}$, select an arbitrary subset
\[
 Q_i(v) \subset N(v) \cap A_i
\]
of size $|Q_i(v)| = C \eps(t) k n^{1/\ell}$. Let~$\Q$ be the set of all paths of the form $(x,u_1, \ldots, u_t)$, generated as follows: For each $i=1,2, \ldots, t$, select $u_i \in Q_i(u_{i-1})$ satisfying\footnote{Recall that $L^{(1)}_\F\big( \cdot \big)$ is a collection of edges of $G$.}
$$u_i \not\in V(S_i)
\qquad\mbox{ and }\qquad
u_{i-1}u_i \not\in L^{(1)}_\F\big( E(S_i) \big),$$
where $S_i$ is the path $(x,u_1, \ldots, u_{i-1})$.

We claim that
\begin{equation}
 \label{eq:initial_Q_is_large}
 |\Q| \ge \frac{3}{4} \big( C\eps(t) kn^{1/\ell} \big)^t.
\end{equation}
To see this, simply recall that $\big| L^{(1)}_\F\big( E(S_i) \big) \big| \le 2^{5\ell} \delta k^{\ell/(\ell-1)}$, by Lemma~\ref{lem:size_of_link}, and so the number of choices for the vertex~$u_i$ is at least
\[
 C \eps(t) k n^{1/\ell} - 2\ell - 2^{5\ell} \delta k^{\ell/(\ell-1)} \, \ge \, \frac{C \eps(t)}{(4/3)^{1/t}} \cdot k n^{1/\ell},
\]
where we used the inequalities $\delta \ll \eps(t)$ and $k \le n^{(\ell-1)/\ell}$.

We now remove some paths from the collection~$\Q$ to produce the collection~$\P$. If there exists $0 \le i < j \le t$ with $(i,j) \ne (0,t)$ and vertices $u \in A_i$ and $v \in A_j$ such that
 \begin{equation}\label{eq:pathalg1}
 \big| \Q_{i,j}[u \to v] \big| \, > \, k^{(j-i-1)\ell/(\ell-1)},
\end{equation}
then choose a path $Q = (x,u_1,\ldots,u_t) \in \Q$ with $u_i = u$ and $u_j = v$, and remove $Q$ from $\Q$. Repeat this until there are no such paths remaining in~$\Q$, and let $\P$ be the set of paths remaining at the end. By construction, $\P$ satisfies conditions~$(ii)$ and~$(iii)$ of Definition~\ref{def:balanced:nbhd}, and moreover it avoids~$\F$, since none of the edges of $G$ are saturated (by~\eqref{eq:Delta1_bounded}), and since we chose $u_i$ so that $u_{i-1}u_i \not\in L^{(1)}_\F\big( E(S_i) \big)$. 

It only remains to prove~\eqref{eq:number:of:paths}; to do so, we will use the fact that $t(x) = t(G)$ to bound the number of paths removed from $\Q$. Given $0 \le i < j \le t$, let us say that an ordered pair of vertices $(u,v)$ is $(i,j)$-{\em unbalanced} if~\eqref{eq:pathalg1} holds (in the original family $\Q$), and set 
$$\R(i,j) \,=\, \big\{ Q = (x,u_1,\ldots,u_j) \in \Q_{0,j} \,:\, (u_i, u_j) \textup{ is $(i,j)$-unbalanced} \big\}.$$
We claim that
\begin{equation}
 \label{eq:size_of_Rij}
 |\R(i,j)| \le \ds\frac{\big( C\eps(t)kn^{1/\ell} \big)^j}{4t^2}
\end{equation}
for every $0 \le i < j \le t$ with $(i,j) \ne (0,t)$. To prove~\eqref{eq:size_of_Rij}, note first that if $s := j - i = 1$, then $\R(i,j) = \emptyset$ (since no pair of vertices $(u,v)$ can be $(i,i+1)$-unbalanced), so we may assume that $s \ge 2$. Next, observe that
$$|\R(i,j)| \le \sum_{u \in A_i} \big| \R(i,j)_{0,i}[x \to u] \big| \cdot \big| \R(i,j)_{i,j}[u \to A_j] \big|,$$
and that $\sum_{u \in A_i} |\R(i,j)_{0,i}[x \to u]| \le \big( C \eps(t)kn^{1/\ell} \big)^i$, since $\R(i,j)_{0,i} \subset \Q_{0,i}$ and so we have at most $C \eps(t)kn^{1/\ell}$ choices at each step. Hence, if $|\R(i,j)| \ge (1/4t^2) \big( C \eps(t)kn^{1/\ell} \big)^j$, then there must exist a vertex $u \in A_i$ such that
\begin{equation}\label{eq:Rpaths:utoBj}
\big| \R(i,j)_{i,j}[u \to A_j] \big| \, \ge \, \frac{1}{4t^2} \cdot \big( C \eps(t) kn^{1/\ell} \big)^s \, \ge \, 2Ct \cdot \eps(s) \big( kn^{1/\ell} \big)^s
\end{equation}
since $2 \le s \le t - 1$, and therefore $C \eps(t)^{s} \ge 8t^3 \cdot \eps(s)$. We will show that $t(u) < t(x)$.

To do so, we will apply Lemma~\ref{lem:paths_to_concentrated_nbd} to the collection $\R(i,j)_{i,j}[u \to A_j]$. Recall that $s \ge 2$, and note that the required bounds on the number of paths and the branching factor follow from~\eqref{eq:Rpaths:utoBj} and the definition of $\Q$ respectively. However, we also require an upper bound on the size of the set
$$B \, := \, \big\{ u_j \in A_j \,:\, \textup{there exists a path } (x,u_1,\ldots,u_j) \in \R(i,j) \textup{ with } u_i = u \big\},$$
that is, the set of end-vertices of the paths in $\R(i,j)$ whose $i$th vertex is $u$. Observe that, since each pair $(u,u_j)$ is unbalanced, we have 
\begin{equation}\label{eq:boundingRju}
|B| \, \le \, \frac{ \big( C \eps(t) k n^{1/\ell} \big)^{j-i} }{k^{(j-i-1)\ell/(\ell-1)} } \, \le \, k^{(\ell-j+i)/(\ell-1)}n^{(j-i)/\ell}.
\end{equation}
By Lemma~\ref{lem:paths_to_concentrated_nbd}, it follows from~\eqref{eq:Rpaths:utoBj} and~\eqref{eq:boundingRju} that $t(u) \le s < t(x)$, as claimed. This contradicts the minimality of $t(x)$, and hence proves~\eqref{eq:size_of_Rij}. 

Finally, it follows from~\eqref{eq:size_of_Rij}, and the branching factor of $\Q$, that at most
$$\sum_{i,j} \big| \R(i,j) \big| \cdot \big( C \eps(t) k n^{1/\ell} \big)^{t-j} \, \le \, \frac{1}{4} \cdot \big( C\eps(t)kn^{1/\ell} \big)^t$$
paths are removed. Combining this with~\eqref{eq:initial_Q_is_large} gives~\eqref{eq:number:of:paths}, as required.
\end{proof}

Before moving on to the meat of the proof -- the construction of the family of cycles $\C$ -- let us note two simple but key properties of balanced neighbourhoods.

\begin{lemma} \label{lem:counting_paths_through_vertices}
Let $x \in V(G)$, and let $(\A,\P)$ be a balanced $t$-neighbourhood of $x$. Let $w \in A_t$, and let $v \in V(G) \setminus \{x,w\}$. There are at most
 \[
 \ell\cdot k^{(t-2)\ell/(\ell-1)}
 \]
paths $P \in \P[x \to w]$ with $v \in V(P)$.
\end{lemma}

\begin{proof}
For each $1 \le j \le t - 1$, we count the number of paths $(x,u_1,\ldots,u_{t-1},w)$ in $\P[x \to w]$ such that $u_j = v$. By property~$(ii)$ of Definition~\ref{def:balanced:nbhd}, we obtain a bound of 
$$\sum_{j = 1}^{t-1} |\P_{0,j}[x \to v]| \cdot |\P_{j,t}[v \to w]| \, \le \, (t-1) \cdot k^{(t-2)\ell/(\ell-1)} \, \le \, \ell \cdot k^{(t-2)\ell/(\ell-1)},$$ 
as claimed.
\end{proof}

\begin{lemma} \label{lem:counting_paths_through_sets}
Let $x \in V(G)$, and let $(\A,\P)$ be a balanced $t$-neighbourhood of $x$. Let $\sigma \subset E(G)$ with $1 \le |\sigma| \le t - 1$. For each $w \in A_t$, there are at most
 \[
 t^t \cdot k^{(t-|\sigma|-1)\ell/(\ell-1)}
 \]
 paths $P \in \P[x \to w]$ with $\sigma \subset E(P)$.
\end{lemma}

\begin{proof}
When $|\sigma| = t-1$ the result is trivial, since there is at most one path $P \in \P[x \to w]$ with $\sigma \subset E(P)$. Therefore, let us assume that $1 \le |\sigma| \le t - 2$. We will count the number of paths in $\P[x \to w]$ with $\sigma \subset E(P)$ as follows. Observe that, if $\sigma \subset E(P)$ for some path $P \in \P$, then $\sigma$ may be decomposed into a sequence of paths between disjoint collections of sets in $\A$. Since property~$(ii)$ of Definition~\ref{def:balanced:nbhd} gives us a bound on the number of ways of travelling between a given vertex of $A_i$ and a given vertex of $A_j$ along a path of $\P$, it is easy to deduce the claimed bound.

To spell out the details, let $\sigma = \sigma_1 \cup \ldots \cup \sigma_r$, where $\sigma_q$ is a path between $u_q \in A_{i(q)}$ and $v_q \in A_{j(q)}$, and $i(q) < j(q) < i(q') < j(q')$ for each $1 \le q < q' \le r$. Note that the sequence $\big( i(1), j(1), \ldots, i(r), j(r) \big)$ might not be unique, since the sets of $\A$ may not be disjoint, but in any case there are at most $t^t$ possible sequences for each set $\sigma$. By~property~$(ii)$ of Definition~\ref{def:balanced:nbhd}, and setting $v_0 = x$, $u_{r+1} = w$, $j(0) = 0$ and $i(r+1) = t$, we have
$$\prod_{q=0}^r \big| \P_{j(q),i(q+1)}[v_q \to u_{q+1}] \big| \, \le \, \big( k^{\ell / (\ell-1)} \big)^{\sum_{q=0}^r \max\{ i(q+1) - j(q) - 1, 0\}}.$$
Since $\sum_{q=0}^r \big( i(q+1) - j(q) \big) = t - |\sigma|$, it follows that the number of paths in~$\P[x \to w]$ containing~$\sigma$ is at most
\[
t^t \cdot k^{(t-|\sigma|-1)\ell/(\ell-1)},
\]
as required. 
\end{proof}

\subsection{Refined \texorpdfstring{$t$}{t}-neighbourhoods}

In this section we will show how to `refine' a balanced $t$-neighbourhood so that each vertex has a sufficiently large neighbourhood, and each vertex of the final set receives many paths, without destroying too many of the paths of $\P$.  

\begin{defn}[Refined $t$-neighbourhood]\label{def:refined:nbhd}
Let $(\B, \Q)$ be a balanced $t$-neighbourhood of $x$. We say that $(\B, \Q)$ form a {\em refined $t$-neighbourhood} of $x$ if the following conditions also hold:
\begin{enumerate}
\item[$(i)$]
For every $i \in \{0,1,\ldots,t-1\}$ and every $u \in B_i$, 
$$|N(u)\cap B_{i+1}| \, \ge \, \eps(t) k n^{1/\ell}.$$
\item[$(ii)$]
For every $v \in B_t$,
$$|N(v)\cap B_{t-1}| \, \ge \, \eps(t)^2 k^{\ell/(\ell-1)}.$$
 \item[$(iii)$]
For every $v \in B_t$,
$$|\Q[x \to v]| \, \ge \, \eps(t)^t k^{(t-1)\ell/(\ell-1)}.$$
\end{enumerate}
\end{defn}

The properties above will allow us to find many cycles through $x$ of the form described earlier (a path out from $x$, a zig-zag path, and a path back to $x$). The following lemma allows us to find a refined neighbourhood inside a balanced neighbourhood. Recall that $t = t(G)$.

\begin{lemma}\label{lem:balanced_to_refined}
If $(\A, \P)$ is a balanced $t$-neighbourhood of $x$, and 
$$|\P| \, \ge \, \frac{C^t}{2} \big( \eps(t) kn^{1/\ell} \big)^t,$$
then there exists a refined $t$-neighbourhood $(\B, \Q)$ of $x$, with $\B \prec \A$ and $\Q \subset \P$.
\end{lemma}

\begin{proof}[Proof of Lemma~\ref{lem:balanced_to_refined}]
Repeatedly perform the following three steps until no further vertices are removed.
\begin{enumerate}
  \item[1.] If there exists $i \in \{1,\ldots,t-1\}$ and a vertex $v \in A_i$ with 
 $$|N(v)\cap A_{i+1}| < \eps(t) kn^{1/\ell},$$
 then remove $v$ from $A_i$, and remove all paths $P = (x,u_1,\ldots,u_t) \in \P$ with $u_i = v$.
\item[2.] If there exists a vertex $v \in A_t$ with
 $$|N(v)\cap A_{t-1}| < \eps(t)^2 k^{\ell/(\ell-1)},$$
 then remove $v$ from $A_t$, and remove all paths $P = (x,u_1,\ldots,u_t) \in \P$ with $u_t = v$.
   \item[3.] If there exists a vertex $v \in A_t$ with 
 $$|\P[x \to v]| < \eps(t)^t k^{(t-1)\ell/(\ell-1)},$$
 then remove $v$ from $A_t$, and remove all paths $P = (x,u_1,\ldots,u_t) \in \P$ with $u_t = v$.
\end{enumerate}
Let~$\B \prec \A$ and $\Q \subset \P$ be the collection of sets and paths at the end of this process.

We claim that $(\B,\Q)$ is a refined $t$-neighbourhood of $x$. Indeed, properties~$(ii)$ and~$(iii)$ of Definition~\ref{def:refined:nbhd}, and also property~$(i)$ for $i \neq 0$, hold by construction, since we would have removed any offending vertex or path. It therefore only remains to prove that property~$(i)$ holds for $i=0$, i.e., that $|N(x) \cap B_1| \ge \eps(t) k n^{1/\ell}$. 

In order to prove this, we will in fact show that 
\begin{equation}\label{eq:ref:numberofpathsinP:repeat}
|\Q| \, \ge \, \frac{C^t}{4} \big( \eps(t) kn^{1/\ell} \big)^t.
\end{equation}
Since, by property~$(iii)$ of Definition~\ref{def:balanced:nbhd}, the branching factor of~$\P$ is at most $C\eps(t)kn^{1/\ell}$, the required bound on the degree of $x$ follows immediately. In order to prove~\eqref{eq:ref:numberofpathsinP:repeat}, we shall consider each of the three steps, showing for each that not too many paths are destroyed.

We claim first that few paths are removed in Steps~1 and~3. Indeed, again using our bound on the branching factor of $\P$, it follows that in Step~1 we remove at most
\[
t \cdot \eps(t) k n^{1/\ell} \cdot \big( C \eps(t) k n^{1/\ell} \big)^{t-1} \, = \, t \cdot C^{t-1} \big( \eps(t) k n^{1/\ell} \big)^t
\]
paths from~$\P$, and in Step~3 we remove at most
\[
 \eps(t)^t k^{(t-1)\ell/(\ell-1)} |A_t| \, \le \, \big( \eps(t) k n^{1/\ell} \big)^t
\]
paths from~$\P$, since $|A_t| \le k^{(\ell-t)/(\ell-1)}n^{t/\ell}$, by property~$(i)$ of Definition~\ref{def:balanced:nbhd}.

Finally, we claim that at most $2 C^{t-1} \big( \eps(t)kn^{1/\ell} \big)^t$ paths are destroyed in Step~2; to prove this we will again use the minimality of $t = t(x)$. Indeed, let $Z \subset A_t$ and $\P(Z)$ denote the collection of vertices and paths (respectively) removed in Step~2, and consider the set 
$$Y \, = \, \big\{ v \in A_{t-1} \,:\, v \textup{ has $(t-1)$-branching factor at least $\eps(t) k n^{1/\ell}$ in } \P(Z) \big\}.$$
Summing the $(t-1)$-branching factors in $\P(Z)$ of the vertices in $Y$, observing that each vertex of $Z$ contributes at most $\eps(t)^2 k^{\ell/(\ell-1)}$ to this sum\footnote{Note that this is true even if the vertex has more than this many neighbours in the original set $A_{t-1}$.}, and recalling that $|Z| \le |A_t| \le k^{(\ell-t)/(\ell-1)} n^{t/\ell}$, we obtain
\begin{equation}\label{eq:Ybound}
|Y| \, \le \, \frac{|Z| \cdot \eps(t)^2 k^{\ell/(\ell-1)}}{\eps(t) k n^{1/\ell}} \, \le \, \eps(t) k^{(\ell - t + 1)/(\ell-1)} n^{(t-1)/\ell}.
\end{equation}
Moreover, by the definition of $Y$ and our bound on the branching number of $\P$, at most $C^{t-1} (\eps(t) k n^{1/\ell})^t$ paths in~$\P(Z)$ have their penultimate vertex in~$A_{t-1} \setminus Y$ and final vertex in~$Z$. Hence, if more than $2 C^{t-1} \big( \eps(t)kn^{1/\ell} \big)^t$ paths were destroyed via the removal of vertices in $Z$, then there exists a set of $C^{t-1}  \big( \eps(t) k n^{1/\ell} \big)^{t}$ paths in $\P(Z)$ whose penultimate vertex is in $Y$, and hence (again using our bound on the branching number) there exists a set of at least
\begin{equation}\label{eq:numberofpathsthroughY}
\frac{C^{t-1}  \big( \eps(t) k n^{1/\ell} \big)^{t}}{C \eps(t)  k n^{1/\ell} }  \, \ge \, 2C t \cdot \eps(t-1) (k n^{1/\ell})^{t-1}
\end{equation} 
paths of length $t-1$ in~$G$ from $x$ to~$Y$, since $\eps(t-1) = \eps(t)^t$ and $2Ct \cdot \eps(t) \le 1$. If $t \ge 3$, then by removing the final edge from each of these paths, we obtain a collection of paths of length $t-1 \ge 2$ from $x$ to $Y$ that, by~\eqref{eq:Ybound} and~\eqref{eq:numberofpathsthroughY}, satisfies the conditions of Lemma~\ref{lem:paths_to_concentrated_nbd}. This implies $t(x) \le t - 1$, which contradicts our assumption that $t(x) = t$. On the other hand, if $t = 2$ then~\eqref{eq:Ybound} states that $|Y| \le \eps(t)  k n^{1/\ell}$, and hence, by our bound on the branching number, there are at most $C \big( \eps(t) k n^{1/\ell} \big)^2$ paths in $\P$ whose penultimate vertex is in $Y$. Thus, in either case, at most $2 C^{t-1} \big( \eps(t)kn^{1/\ell} \big)^t$ paths are destroyed in Step~2, as claimed.

Putting the pieces together, and recalling that $C \ge 10\ell$, it follows that at most
$$\Big( (t + 2) C^{t-1} + 1 \Big)  \big( \eps(t) k n^{1/\ell} \big)^t \, \le \, \frac{C^t}{4} \big( \eps(t) k n^{1/\ell} \big)^t$$
paths were removed in Steps~1,~2 and~3 combined, and hence~\eqref{eq:ref:numberofpathsinP:repeat} holds. By the comments above, it follows that $(\B,\Q)$ is a refined $t$-neighbourhood of~$x$, as required. 
\end{proof}

\subsection{Finding cycles in refined \texorpdfstring{$t$}{t}-neighbourhoods}\label{sec:finding:cycles}

We are now ready to complete the proof of the key proposition.

\begin{proof}[Proof of Proposition~\ref{prop:finding:cycles}]
Fix an arbitrary $x \in V(G)$ with $t(x) = t(G) = t$. By Lemma~\ref{prop:RNF:exists}, there exists a balanced $t$-neighbourhood $(\A,\P)$ of~$x \in V(G)$, avoiding $\F=\F(\HH)$. Applying Lemma~\ref{lem:balanced_to_refined}, we obtain a refined $t$-neighbourhood $(\B, \Q)$ with $\B \prec \A$ and $\Q \subset \P$. 

To find a copy $H$ of $C_{2\ell}$ in~$G$ that is not already in~$\HH$ and with $\HH \cup \{H\}$ good, we will find a large collection of $2\ell$-cycles, each cycle containing an edge between~$x$ and~$B_1$ and avoiding all saturated sets of edges. Note that, since $|B_1| \le kn^{1/\ell}$ by property~$(i)$ of Definition~\ref{def:balanced:nbhd}, it is sufficient to find such a collection of cycles~$\C$ with
\begin{equation}
 \label{eq:size_of_C}
 |\C| \, \ge \, \eps(1)^{3\ell-1} k^{2\ell} n.
\end{equation}
Indeed, by the pigeonhole principle and our bound~\eqref{eq:deg1bound} on the degree of a single edge in $\HH$, it follows from~\eqref{eq:size_of_C} that at least one of these cycles will not already be in~$\HH$. Since (we will ensure that) no cycle in $\C$ contains a saturated set of edges, we will be able to add this cycle to~$\HH$ and obtain a hypergraph that is still good, as required.

As noted earlier, each cycle in $\C$ will be formed by paths~$P$ and~$Q$, constructed as follows: $P$, of length $2\ell - t$, goes from~$x$ to~$B_t$, and then alternates between $B_t$ and $B_{t-1}$, finishing at some vertex $v_{2\ell-t} \in B_t$, and $Q \in \Q[x \to v_{2\ell-t}]$ is chosen so that $P \cup Q$ does not contain any forbidden set. For technical reasons (see Claim~\ref{claim:m_of_j}, below), it is important that there are not too many choices for the path $P$. Therefore, let us first choose sets
$$X_i(u) \subset N(u) \cap B_{i+1} \quad \text{with} \quad |X_i(u)| = \eps(t) k n^{1/\ell}$$ 
for each $i \in \{0,\ldots,t-1\}$ and each $u \in B_i$, and 
$$X_t(u) \subset N(u) \cap B_{t-1}\quad \text{with} \quad |X_t(u)| = \eps(t)^2 k^{\ell/(\ell-1)}$$ 
for each $u \in B_t$. (These exist by properties~$(i)$ and~$(ii)$ of Definition~\ref{def:refined:nbhd}.) We will choose the vertices of the path $P$ from the sets $X_i(u)$.  

To be precise, we perform the following algorithm:

\begin{Alg}
Let~$\C$ be the set of cycles obtained via the following process.
\begin{enumerate}
\item[0.] Set $\C = \emptyset$. Now repeat the following two steps until STOP.\smallskip

 \item[1.] If possible, generate a path $(v_0, v_1, \ldots, v_{2\ell-t})$ not previously generated in this step as follows. Let $v_0 = x$, and define
 \[
  s(i) = \begin{cases}
           i &\mbox{if } i \in \{0, 1, \ldots, t\}, \\
           t-1 &\mbox{if } i \in \{t+1, t+3, \ldots, 2\ell-t-1\}, \\
           t &\mbox{if } i \in \{t+2, t+4, \ldots, 2\ell-t\}.
          \end{cases}
 \]
Now, for $i = 0,\ldots, 2\ell-t-1$, select $v_{i+1}$ from $X_{s(i)}(v_{i}) \subset N(v_{i}) \cap B_{s(i+1)}$ such that
 \begin{equation}
 \label{eq:path_extn_condn}
  v_{i+1} \not \in \big\{ v_0, \ldots, v_i \big\}  \qquad \mbox{and} \qquad  v_iv_{i+1} \not\in L^{(1)}_\F\big( E(P_i) \big),
 \end{equation}
where $P_i$ is the path $(v_0, \ldots, v_i)$. Set $P = P_{2\ell-t}$.  \\
Otherwise (that is, if no such path~$P$ exists that has not already been generated), then STOP.\smallskip

 \item[2.] Let $\Q(P) \subset \Q[x \to v_{2\ell-t}]$ be an arbitrary set of exactly $\eps(t)^t k^{(t-1)\ell/(\ell-1)}$ paths ending at $v_{2\ell-t}$, and let $\Q'(P) \subset \Q(P)$ denote the collection of paths $Q \in \Q(P)$ which use no vertex of $\{ v_1, \ldots, v_{2\ell-t-1} \}$ and avoid $L_\F ( E(P) )$.\\
\noindent For each path $Q \in \Q'(P)$, join the paths~$P$ and~$Q$ to form a cycle and add this to~$\C$.
\end{enumerate}
\end{Alg}

Since $v_iv_{i+1} \not\in L^{(1)}_\F\big( E(P_i) \big)$ for each $0 \le i \le 2\ell-t-1$, since the paths in $\Q(P)$ avoid $L_\F ( E(P) )$, and since $\Q \subset \P$ avoids $\F$, it follows that no member of $\C$ contains any saturated set of edges. In order to complete the proof of of Proposition~\ref{prop:finding:cycles}, it will therefore suffice to show that~\eqref{eq:size_of_C} holds.

\newcounter{ClaimsInCycleFinding}

\medskip
\refstepcounter{ClaimsInCycleFinding}
\noindent \textbf{Claim \arabic{ClaimsInCycleFinding}:}\label{claim:many_choices_for_P} There are at least 
$$\frac{1}{2} \cdot \big( \eps(t)k n^{1/\ell} \big)^\ell \cdot \big( \eps(t)^2 k^{\ell/(\ell-1)} \big)^{\ell-t}$$ 
choices for the path~$P$ in Step~1.

\begin{proof}[Proof of Claim~\ref{claim:many_choices_for_P}]
We will show that at most $2\ell + 2^{5\ell} \delta k^{\ell/(\ell-1)}$ choices are excluded for each vertex. To see this, note first that at most $2\ell$ choices are excluded by the condition that $v_{i+1} \not \in \big\{ v_0, \ldots, v_i \big\}$. Moreover, by Lemma~\ref{lem:size_of_link} we have
\[
 \big| L_\F^{(1)}\big( E(P_i) \big) \big| \, \le \, 2^{2\ell + e(P_i) + 1} \cdot \delta k^{\ell/(\ell-1)} \, \le \, 2^{5\ell} \delta k^{\ell/(\ell-1)},
\]
as required. Hence, if $i \in \{1,\ldots,t-1\} \cup \{t+1, t+3, \ldots, 2\ell-t-1\}$, then there are at least
\[
 \eps(t) kn^{1/\ell} - \Big( 2\ell + 2^{5\ell} \delta k^{\ell/(\ell-1)} \Big) \, \ge \, \frac{1}{2^{1/2\ell}} \cdot \eps(t) kn^{1/\ell}
\]
choices for $v_{i+1}$, where the last inequality follows since $k \le n^{(\ell-1)/\ell}$ and $\delta \ll \eps(t)^2$. Similarly, if $i \in \{t, t+2, \ldots, 2\ell-t-2\}$, then there are at least
\[
 \eps(t)^2 k^{\ell/(\ell-1)} - \Big( 2\ell + 2^{5\ell} \delta k^{\ell/(\ell-1)} \Big) \, \ge \, \frac{1}{2^{1/2\ell}} \cdot \eps(t)^2 k^{\ell/(\ell-1)}
\]
choices for $v_{i+1}$, again using the fact that~$\delta \ll \eps(t)^2$. It follows immediately that the total number of choices for~$P$ is at least
$$\frac{1}{2} \cdot \big( \eps(t)k n^{1/\ell} \big)^\ell \cdot \big( \eps(t)^2 k^{\ell/(\ell-1)} \big)^{\ell-t}$$ 
as claimed.
\end{proof}

Let $\D$ denote the collection of paths $P$ generated in Step~1 for which there are at least 
$$\frac{1}{4} \cdot \eps(t)^t k^{(t-1)\ell/(\ell-1)}$$ 
paths $Q \in \Q(P)$ with $E(Q) \in L_\F(E(P))$. We will show later that $|\D|$ is not too large. 

The next claim shows that if the path chosen in Step~1 satisfies $P \not\in \D$, then we have many choices for the path $Q$ in Step~2. 

\medskip
\refstepcounter{ClaimsInCycleFinding}
\label{claim:Q_prime_is_large}
\noindent \textbf{Claim \arabic{ClaimsInCycleFinding}:} Let $P$ be a path chosen in Step~1. If $P \not\in \D$, then 
$$|\Q'(P)| \, \ge \, \frac{1}{2} \cdot \eps(t)^t k^{(t-1)\ell/(\ell-1)}.$$ 

\begin{proof}[Proof of Claim~\ref{claim:Q_prime_is_large}]
Since $|\Q(P)| = \eps(t)^t k^{(t-1)\ell/(\ell-1)}$, we are required to bound $|\Q(P) \setminus \Q'(P)|$ from above, that is, to bound the number of paths in $\Q(P)$ that either contain a vertex of $P$, or fail to avoid $L_\F(E(P))$.

To do so, note first that, by Lemma~\ref{lem:counting_paths_through_vertices}, the number of paths in $\Q(P)$ which contain some vertex of $P$ is at most
$$2\ell^2 \cdot k^{(t-2)\ell/(\ell-1)} \, \le \, \eps(t)^{t+1} k^{(t-1)\ell/(\ell-1)},$$
 where the last inequality follows since $k \ge k_0 = \eps(1)^{-3\ell}$. So let $\sigma \in L_\F(E(P))$, and consider the paths in $\Q(P)$ which contain $\sigma$. Suppose first that $1 \le | \sigma | \le t - 1$. Then, by Lemma~\ref{lem:counting_paths_through_sets}, the number of paths in~$\Q(P)$ containing~$\sigma$ is at most
 \[
t^t k^{(t-|\sigma|-1)\ell/(\ell-1)}, 
\]
and recall that, by Lemma~\ref{lem:size_of_link}, we have
 \[
  \big| L_\F^{(|\sigma|)}( E(P) ) \big| \, \le \, 2^{5\ell} \big( \delta k^{\ell/(\ell-1)} \big)^{|\sigma|}.
 \]
Therefore, multiplying the above expressions, it follows that the number of paths in $\Q(P)$ which contain some $\sigma \in L_\F(E(P))$ with $1 \le |\sigma| \le t - 1$ is at most 
$$(2\ell)^{5\ell} \cdot \delta k^{(t-1)\ell/(\ell-1)} \, \le \, \eps(t)^{t+1} k^{(t-1)\ell/(\ell-1)},$$
where the last inequality holds since $\delta = \eps(1)^{3\ell}$. On the other hand, since $P \not\in \D$, there are at most 
$$\frac{1}{4} \cdot \eps(t)^t k^{(t-1)\ell/(\ell-1)}$$ 
paths in $\Q(P)$ which contain (and are therefore equal to) some $\sigma \in L_\F(E(P))$ with $|\sigma| = t$. Summing these three bounds, it follows that 
$$|\Q(P) \setminus \Q'(P)| \, \le \, \frac{1}{2} \cdot \eps(t)^t k^{(t-1)\ell/(\ell-1)},$$
as claimed. 
\end{proof}

Finally, we need to show that $\D$ is not too large. In order to do so, we will need two easy but technical claims. The first bounds the number of paths in $\D$ ending at a given vertex and containing a given set of edges.

\medskip
\refstepcounter{ClaimsInCycleFinding}
\label{claim:m_of_j}
\noindent \textbf{Claim \arabic{ClaimsInCycleFinding}:} Given a set of edges $J \subset E(G)$ of size $|J| = j$, and a vertex $v \in B_t$, there are at most
\begin{equation*}
 m(j) \,:=\, \begin{cases}
  (2\ell)^{2\ell} \cdot \big( \eps(t) kn^{1/\ell} \big)^{\ell-1} \cdot \big( \eps(t)^2 k^{\ell/(\ell-1)} \big)^{\ell-t-j} &\mbox{if } 1 \le j \le \ell-t, \\
  (2\ell)^{2\ell} \cdot \big( \eps(t) k n^{1/\ell} \big)^{2\ell-t-j-1} &\mbox{if } \ell - t < j < 2\ell - t
 \end{cases}
\end{equation*}
paths $P \in \D$ ending at $v$ with $J \subset E(P)$.

\begin{proof}[Proof of Claim~\ref{claim:m_of_j}]
Note that we have at most $(2\ell)^{2\ell}$ choices for the positions of the edges of $J$ in the path $P$. So let's fix such a choice, and count the corresponding paths.

To do so, it suffices to observe that at least $j+2$ of the vertices of $P$ are fixed (these are the endpoints of edges in $J$, and the endpoints $x$ and $v$ of $P$), and that $kn^{1/\ell} \ge k^{\ell/(\ell-1)}$. The bound in the case $j > \ell - t$ now follows immediately, since we have at most $\eps(t) k n^{1/\ell}$ choices for each not-yet-chosen vertex $v_i$. Moreover, the bound in the case $j \le \ell - t$ also follows easily, we have $\eps(t) k n^{1/\ell}$ choices for (at most) $\ell - 1$ of the not-yet-chosen vertices, and at most $\eps(t)^2 k^{\ell/(\ell-1)}$ choices for each of the remaining vertices. 
\end{proof}

The key property of the bound $m(j)$ is that it satisfies
\begin{equation}\label{eq:mj:bound}
m(j) \cdot \big( \delta k^{\ell / (\ell - 1)} \big)^j \, \le \, \big( \eps(t) kn^{1/\ell} \big)^{\ell-1} \cdot \big( \eps(t)^2 k^{\ell/(\ell-1)} \big)^{\ell-t}
\end{equation}
for every $1 \le j < 2\ell - t$. We shall use the inequality~\eqref{eq:mj:bound} in the proof of Claim~\ref{claim:D_is_small}, below. Our second technical claim gives a bound on the number of pairs of edge sets $(J,Q)$, where $Q \in \Q(P)$ for some $P \in \D$, and $E(Q) \cup J \in \F$.

\medskip
\refstepcounter{ClaimsInCycleFinding}
\label{claim:Jexists}
\noindent \textbf{Claim \arabic{ClaimsInCycleFinding}:} For some $1 \le j < 2\ell - t$, there exist at least
\begin{equation}\label{eq:countingRJ:lower}
2^{-3\ell} \eps(t)^t k^{(t-1)\ell/(\ell-1)} \cdot \frac{|\D|}{2\ell \cdot m(j)}
\end{equation}
distinct pairs $(J,Q)$ with the following properties:
\begin{itemize}
\item[$(a)$] $Q \in \Q(P)$ for some path $P \in \D$, 
\item[$(b)$] $J$ is a set of $j$ edges of $G$ disjoint from $E(Q)$,
\item[$(c)$] $E(Q) \cup J \in \F$.
\end{itemize}

\begin{proof}[Proof of Claim~\ref{claim:Jexists}]
Recall that for each path $P \in \D$, there are at least $\frac{1}{4} \cdot \eps(t)^t k^{(t-1)\ell/(\ell-1)}$ paths $Q \in \Q(P)$ with $E(Q) \in L_\F(E(P))$. By the pigeonhole principle, it follows that for each such path $P$, there exists a set $\emptyset \ne f(P) \subset E(P)$ such that there are at least
\begin{equation}\label{eq:countingRs}
 2^{-3\ell} \eps(t)^t k^{(t-1)\ell/(\ell-1)}
\end{equation}
paths $Q \in \Q(P)$, each of which is disjoint from $f(P)$ and such that  $E(Q) \cup f(P) \in \F$. By another application of the pigeonhole principle, there exists $1 \le j < 2\ell - t$ for which there are at least $|\D|/2\ell$ paths $P \in \D$ with $|f(P)| = j$. 

We claim that there are at least
\begin{equation}\label{eq:Dover2ellmj}
\frac{|\D|}{2\ell \cdot m(j)}
\end{equation}
distinct pairs $(J,v)$ such that $v \in B_t$, $|J| = j$ and $f(P) = J$ for some path $P \in \D$ which ends at $v$. Indeed, since $f(P) \subset E(P)$ for each $P \in \D$, it follows from Claim~\ref{claim:m_of_j} that for each pair $(J,v)$ with $v \in B_t$ and $|J| = j$, there are at most $m(j)$ paths $P \in \D$ ending at~$v$ with $f(P) = J$. Thus, by the observation above, there must be at least $|\D| / (2\ell \cdot m(j))$ such pairs with at least one such path, as claimed.

Finally, for each such pair $(J,v)$ choose a path $P \in \D$ which ends at $v$ with $f(P) = J$, and recall that there are~\eqref{eq:countingRs} paths $Q \in \Q(P)$ with $E(Q) \cup J \in \F$, each of which is disjoint from $J$. Since each $Q \in \Q(P)$ has the same endpoint as $P$, it follows that the pairs $(J,Q)$ thus obtained are all different, and hence the claim follows. 
\end{proof}

We are finally ready to bound $|\D|$. We shall do so by showing that if $\D$ were too large, then some edge of~$G$ (more precisely, some edge between $x$ and $B_1$) would be contained in more than $\Delta^{(1)}(k, n)$ members of $\HH$, contradicting our assumption that~$\HH$ is good. 

\medskip
\refstepcounter{ClaimsInCycleFinding}
\label{claim:D_is_small}
\noindent \textbf{Claim \arabic{ClaimsInCycleFinding}:} $|\D| \le \ds\frac{1}{4} \cdot \big( \eps(t)k n^{1/\ell} \big)^\ell \cdot \big( \eps(t)^2 k^{\ell/(\ell-1)} \big)^{\ell-t}$.

\begin{proof}[Proof of Claim~\ref{claim:D_is_small}]
In order to deduce that some edge between $x$ and $B_1$ is already contained in too many cycles, recall first that if $|E(Q) \cup J| = t + j$ and $E(Q) \cup J \in \F$ then
\[
 d_\HH\big( E(Q) \cup J \big) \, \ge \, \lfloor \Delta^{(t+j)}(k, n) \rfloor \, \ge \, \frac{1}{2} \cdot \frac{\Delta^{(1)}(k, n)}{\big( \delta k^{\ell / (\ell-1)} \big)^{t+j-1}}.
\]
On the other hand, every path~$Q$ with $Q \in \Q(P)$ for some $P \in \D$ contains an edge of~$G$ between~$x$ and $B_1$. Thus, noting that each member of~$\HH$ contains $E(Q) \cup J$ for at most $2^{4\ell}$ pairs $(J,Q)$, it follows from Claim~\ref{claim:Jexists} that $\HH$ contains at least
\[
2^{-3\ell} \eps(t)^t k^{(t-1)\ell/(\ell-1)} \cdot \frac{|\D|}{2\ell \cdot m(j)} \cdot \frac{\Delta^{(1)}(k, n)}{2^{4\ell+1} \big( \delta k^{\ell / (\ell-1)} \big)^{t+j-1}}
\]
cycles, each containing some edge between $x$ and $B_1$. Since $t \ge 2$ and $\delta = \eps(1)^{3\ell}$, this is at least
\[
\frac{4}{\eps(t)} \cdot \frac{|\D| \cdot \Delta^{(1)}(k, n)}{m(j) \cdot \big( \delta k^{\ell / (\ell-1)} \big)^j} \, \ge \, \frac{4}{\eps(t)} \cdot \frac{|\D| \cdot \Delta^{(1)}(k, n)}{\big( \eps(t) kn^{1/\ell} \big)^{\ell-1} \cdot \big( \eps(t)^2 k^{\ell/(\ell-1)} \big)^{\ell-t}},
\]
where the second inequality follows from~\eqref{eq:mj:bound}. Now, since $|B_1| \le kn^{1/\ell}$, by an application of the pigeonhole principle there is an edge $e = xu \in E(G)$ with $u \in B_1$ such that  
\[
d_\HH(e) \, \ge \, 4 \cdot \frac{|\D| \cdot \Delta^{(1)}(k, n)}{\big( \eps(t) kn^{1/\ell} \big)^{\ell} \cdot \big( \eps(t)^2 k^{\ell/(\ell-1)} \big)^{\ell-t}} \, > \, \Delta^{(1)}(k, n)
\]
if $|\D| > \frac{1}{4} \cdot \big( \eps(t)k n^{1/\ell} \big)^\ell \cdot \big( \eps(t)^2 k^{\ell/(\ell-1)} \big)^{\ell-t}$, contradicting our assumption that $\HH$ is good.
\end{proof}

It is now easy to deduce~\eqref{eq:size_of_C}. Indeed, multiplying the number of choices in Step~1 for a path~$P \not\in \D$, by the number of choices in Step~2 for the return path $Q$, assuming $P \not\in \D$, it follows from Claims~\ref{claim:many_choices_for_P},~\ref{claim:Q_prime_is_large} and~\ref{claim:D_is_small} that 
\[
|\C| \, \ge \, \frac{1}{2} \cdot \bigg( \frac{1}{4} \cdot \big( \eps(t)k n^{1/\ell} \big)^\ell \cdot \big( \eps(t)^2 k^{\ell/(\ell-1)} \big)^{\ell-t} \bigg) \bigg( \frac{1}{2} \cdot \eps(t)^t k^{(t-1)\ell/(\ell-1)} \bigg)
  \, \ge \, \eps(t)^{3\ell-1} k^{2\ell} n,
\]
where the initial factor of $1/2$ is because we may have counted each cycle twice. In particular, this implies~\eqref{eq:size_of_C}, and thus completes the proof of Proposition~\ref{prop:finding:cycles}.
\end{proof}

\subsection{Proof of Theorem~\ref{thm:cycle:hypergraph}}\label{ESconj:proofSec}

Finally, let us deduce Theorem~\ref{thm:cycle:hypergraph} from Proposition~\ref{prop:finding:cycles}. 

\begin{proof}[Proof of Theorem~\ref{thm:cycle:hypergraph}]
Starting with~$\HH = \emptyset$, we repeatedly apply Proposition~\ref{prop:finding:cycles} to find a copy $H \not\in \HH$ of $C_{2\ell}$ in $G$ such that $\HH \cup \{H\}$ is good, and add this to~$\HH$. We may continue in this fashion until $e(\HH) \ge \delta k^{2\ell} n^2$, giving a hypergraph that satisfies condition~$(a)$ of the theorem. Since $\HH$ is good with respect to $(\delta,k,\ell,n)$, and
$$\Delta^{(j)}(k,n) \, = \, \frac{k^{2\ell-1} n^{1 - 1  / \ell}}{ \big( \delta k^{\ell / (\ell-1)} \big)^{j-1}} \, \le \, \frac{1}{\delta^{2\ell}} \cdot k^{2\ell - j - \frac{j - 1}{\ell-1}} n^{1 - 1/\ell}$$
for every $j \in [2\ell-1]$ and $k \le n^{1 - 1/\ell}$, it follows that~$\HH$ also satisfies condition~$(b)$ of the theorem, as required.
\end{proof}

\section{Hypergraph containers}\label{sec:containers}

In this section we recall the basic definitions relating to hypergraph containers, and state the theorem from~\cite{BMS,ST} that we will use in order to deduce Theorems~\ref{thm:main} and~\ref{thm:cycle:containers} from Theorem~\ref{thm:cycle:hypergraph}. In order to do so, we will need to define a parameter $\delta(\HH,\tau)$ which (roughly speaking) measures the `uniformity' of the hypergraph $\HH$. 

\begin{defn}\label{def:tau}
Given an $r$-uniform hypergraph $\HH$, define the \emph{co-degree function} of $\HH$ 
$$\delta(\HH,\tau) \, = \, \frac{1}{e(\HH)} \,\sum_{j=2}^r \,\frac{1}{\tau^{j-1}} \sum_{v \in V(\HH)} d^{(j)}(v),$$
where
$$d^{(j)}(v) \, = \, \max\big\{ d_\HH(\sigma) \, : \, v \in \sigma \subset V(\HH) \textup{ and }|\sigma| = j \big\}$$
denotes the maximum degree in $\HH$ of a $j$-set containing $v$. 
\end{defn}

We remark that we have removed some extraneous constants from the definition in~\cite{ST}, since these do not affect the formulation of the theorem below. 

The following theorem is an easy consequence of~\cite[Theorem~6.2]{ST}, see also~\cite[Proposition~3.1]{BMS}. We remark that we may take $\delta_0(r)$ to be roughly $r^{-2r}$. 

\begin{thm}\label{thm:coveroff}
Let $r \ge 2$ and let $0 < \delta < \delta_0(r)$ be sufficiently small. Let $\HH$ be an $r$-graph with~$N$ vertices, and suppose that $\delta(\HH,\tau) \le \delta$ for some $0 < \tau < 1/2$. Then there exists a collection $\C$ of at most
  $$\exp\bigg( \frac{\tau \log( 1/\tau ) N}{\delta} \bigg)$$
 subsets of $V(\HH)$ such that
 \begin{itemize}
 \item[$(a)$] for every independent set $I$ there exists a set $C \in \C$ with $I \subset C$,\smallskip
 \item[$(b)$] $e\big( \HH[C] \big) \le \big(1 - \delta \big) e(\HH)$ for every $C \in \C$.
 \end{itemize}
\end{thm}

We will refer to the collection $\C$ as the \emph{containers} of $\HH$, since, by~$(a)$, every independent set is contained in some member of $\C$. The reader should think of $V(\HH)$ as being the \emph{edge} set of some underlying graph $G$, and $E(\HH)$ as encoding (some subset of) the copies of a forbidden graph $H$ in $G$. Thus every $H$-free subgraph of $G$ is an independent set of $\HH$. 

\subsection{How to apply the container theorem}

In order to motivate the (slightly technical) details of the proof below, let us first give a brief outline of how Theorem~\ref{thm:coveroff} can be used, in conjunction with a balanced supersaturation theorem like Theorem~\ref{thm:cycle:hypergraph}, to count $H$-free graphs on $n$ vertices. Starting with a single container (the complete graph), we repeat the following two steps until all containers (each of which is just a graph on $n$ vertices) have sufficiently few edges. First, we choose a container $C$ in the current collection, and use the balanced supersaturation theorem to bound the co-degree function $\delta(\HH,\tau)$ inside $C$. Theorem~\ref{thm:coveroff} then allows us to replace $C$ by a (not too large) family of (slightly smaller) containers. Using conditions~$(a)$ and~$(b)$ of the theorem above, we can bound the total number of containers produced in this process.

To be a little more precise, let $H = C_{2\ell}$ and assume that the container we are currently considering corresponds to a graph $G$ with $k n^{1 + 1/\ell}$ edges. Theorem~\ref{thm:cycle:hypergraph} provides us with a family of cycles $\HH$ that can be thought of as a $2\ell$-uniform hypergraph on vertex set $E(G)$. We will show (see Proposition~\ref{prop:containers_for_graphs}, below), using our bounds on the size and degrees of $\HH$, that if $\tau = k^{- (1 + 2\eps)}$, then $\delta(\HH,\tau) = o(1)$ as $k \to \infty$. Applying Theorem~\ref{thm:coveroff} to $\HH$, we obtain a collection of at most $\exp\big( k^{-\eps} n^{1 + 1/\ell} \big)$ containers for the $C_{2\ell}$-free subgraphs of $G$, each containing slightly fewer members of $\HH$ than $G$, and hence (using the uniformity of $\HH$), with slightly fewer edges than $G$. It follows that, after about $\log n$ steps, our containers will have only $O(n^{1 + 1/\ell})$ edges and the process will stop. Crucially, because the number of containers produced in a given step is a (rapidly) decreasing function of $k$, the last $O(1)$ steps dominate the count, see the proofs of Theorems~\ref{thm:cycle:containers:turan} and~\ref{thm:cycle:containers:turan:refined}, below.

\section{Counting \texorpdfstring{$H$}{H}-free graphs}\label{sec:proof}

In this section we will prove Theorem~\ref{thm:cycle:containers}, deduce Theorem~\ref{thm:main} and Theorem~\ref{cor:fewwithfewedges}, and prove Proposition~\ref{prop:conj:implies:thm}. In fact, we will prove the following generalization of Theorem~\ref{thm:cycle:containers}, which will also prove useful in studying the Tur\'an problem on the random graph.

\begin{thm}\label{thm:cycle:containers:turan}
For each $\ell \ge 2$, there exists a constant $C = C(\ell)$ such that the following holds for all sufficiently large $n,k \in \N$ with $k \le n^{(\ell-1)^2 / \ell(2\ell-1)} / (\log n)^{\ell-1}$. There exists a collection~$\G_\ell(n,k)$ of at most
$$\exp\Big( C k^{-1/(\ell-1)} n^{1 + 1/\ell} \log k \Big)$$
graphs on vertex set~$[n]$ such that
$$e(G) \le k n^{1+1/\ell}$$ 
for every $G \in \G_\ell(n,k)$, and every $C_{2\ell}$-free graph is a subgraph of some~$G \in \G_\ell(n,k)$. 
\end{thm}

As we will see below, the bounds in the statement above are probably close to best possible, cf. Section~\ref{sec:lowerbounds}. Note that Theorem~\ref{thm:cycle:containers} follows immediately from Theorem~\ref{thm:cycle:containers:turan} by choosing $k$ to be a sufficiently large constant so that $k^{-1/(\ell-1)} \log_2 k < \delta / C$. 

We begin the proof of Theorem~\ref{thm:cycle:containers:turan} by using Theorems~\ref{thm:cycle:hypergraph} and~\ref{thm:coveroff}, as outlined in the previous section, to prove the following container result for $C_{2\ell}$-free graphs.

\begin{prop} \label{prop:containers_for_graphs}
For every $\ell \ge 2$, there exist $k_0 \in \N$ and $\eps > 0$ such that the following holds for every $k \ge k_0$ and every $n \in \N$. Given a graph~$G$ with~$n$ vertices and $k n^{1+1/\ell}$ edges, there exists a collection~$\C$ of at most
$$\exp\bigg( \frac{n^{1 + 1/\ell}}{\eps}  \cdot \max \Big\{ k^{-1/(\ell-1)} \log k, \, n^{-(\ell-1) / \ell(2\ell-1)} \log n \Big\}  \bigg)$$
subgraphs of~$G$ such that:
\begin{itemize}
 \item[$(a)$] Every $C_{2\ell}$-free subgraph of~$G$ is a subgraph of some $C \in \C$, and\smallskip
 \item[$(b)$] $e(C) \le (1 - \eps) e(G)$ for every $C \in \C$.
 \end{itemize}
\end{prop}

\begin{proof}
Given $\ell \ge 2$, let $\delta > 0$ and $k_0 \in \N$ be the constants given by Theorem~\ref{thm:cycle:hypergraph}, and assume that $\delta$ is sufficiently small so that Theorem~\ref{thm:coveroff} holds with $r = 2\ell$. 
Let $G$ be a graph as described in the proposition, and apply Theorem~\ref{thm:cycle:hypergraph} to~$G$. We obtain a $2\ell$-uniform hypergraph~$\HH$ on vertex set~$E(G)$, satisfying\footnote{We recall that the function $\Delta^{(j)}(k,n)$ was defined in~\eqref{def:Delta}.}:
\begin{itemize}
 \item[$(i)$] $e(\HH) \ge \delta k^{2\ell} n^2$,\smallskip
 \item[$(ii)$] $d_\HH(\sigma) \le \Delta^{(|\sigma|)}(k,n)$ 
 for every $\sigma \subset V(\HH)$ with $1 \le |\sigma| \le 2 \ell-1$,
\end{itemize}
such that each of the edges of $\HH$ corresponds to a copy of~$C_{2\ell}$ in~$G$. We will show that if 
$$\frac{1}{\tau} = \delta^4 k \cdot \min\big\{ k^{1/(\ell-1)}, n^{(\ell-1) / \ell ( 2\ell-1)} \big\},$$ 
then it follows from~$(i)$ and~$(ii)$ that $\delta(\HH,\tau) \le \delta$. Indeed, since $v(\HH) = e(G) = kn^{1+1/\ell}$, we have
\begin{align*}
\delta(\HH,\tau) & \, = \, \frac{1}{e(\HH)} \,\sum_{j=2}^{2\ell} \,\frac{1}{\tau^{j-1}} \sum_{v \in V(\HH)} d^{(j)}(v) \\
& \, \le \, \frac{v(\HH)}{e(\HH)} \Bigg[ \sum_{j=2}^{2\ell-1} \Big( \delta^4 k^{\ell/(\ell-1)} \Big)^{j-1} \Delta^{(j)}(k,n) + \Big( \delta^4 k \cdot n^{(\ell-1) / \ell ( 2\ell-1)} \Big)^{2\ell-1} \Bigg]\\
& \, \le \, \frac{1}{\delta k^{2\ell-1} n^{(\ell - 1)/\ell}} \Bigg[ \sum_{j=2}^{2\ell-1}  \frac{\big( \delta^4 k^{\ell/(\ell-1)} \big)^{j-1} \cdot k^{2\ell-1} n^{(\ell - 1)/\ell}}{\big( \delta k^{\ell/(\ell-1)}\big)^{j-1}} + \big( \delta^4 k \big)^{2\ell-1} n^{(\ell-1) / \ell} \Bigg]\\
& \, \le \, \sum_{j=2}^{2\ell-1} \delta^{3j-4} + \delta^3 \, \le \, \delta,
\end{align*}
as required, where we used the bounds~$d_\HH(\sigma) \le \Delta^{(|\sigma|)}(k,n)$ and $1/\tau \le \delta^4 k^{\ell/(\ell-1)}$ when $2 \le |\sigma| \le 2\ell-1$, and the bounds $d_\HH(\sigma) \le 1$ and $1/\tau \le \delta^4 k n^{(\ell-1) / \ell ( 2\ell-1)}$ when $|\sigma| = 2\ell$. Thus, applying Theorem~\ref{thm:coveroff}, and setting $\eps = \delta^6$, we obtain a collection~$\C$ of at most  
$$\exp\bigg( \frac{\tau \log( 1/\tau ) N}{\delta} \bigg) \, \le \, \exp\bigg( \frac{n^{1 + 1/\ell}}{\eps}  \cdot \max \Big\{ k^{-1/(\ell-1)} \log k, \, n^{-(\ell-1) / \ell(2\ell-1)} \log n \Big\} \bigg)$$
 subsets of $V(\HH) = E(G)$ such that:
 \begin{itemize}
 \item[$(a)$] Every $C_{2\ell}$-free subgraph of~$G$ is a subgraph of some $C \in \C$, and \smallskip
 \item[$(b')$] $e\big( \HH[C] \big) \le \big(1 - \delta \big) e(\HH)$ for all $C \in \C$.
 \end{itemize}
It only remains to show that condition~$(b')$ implies that $e(C) \le (1-\eps) e(G)$ for every $C \in \C$. To prove this, for each $C \in \C$ set 
\[
 \D(C) \, = \, E(\HH) \setminus E(\HH[C]) \, = \, \big\{ e \in E(\HH) \,:\, v \in e \mbox{ for some } v \in V(\HH) \setminus C \big\},
\]
and recall that $d_\HH(v) \le e(\HH) / \big( \delta k n^{1 + 1/\ell} \big)$ for every $v \in V(\HH)$, by~$(i)$ and~$(ii)$. Therefore,
\[
 |\D(C)| \, \le \, \frac{e(\HH)}{\delta k n^{1 + 1/\ell}} \cdot |E(G) \setminus C|.
\]
On the other hand, we have $|\D(C)| = e(\HH) - e(\HH[C]) \ge \delta e(\HH)$, by condition~$(b')$, and so
$$ |E(G) \setminus C| \, \ge \, \delta^2 k n^{1+1/\ell},$$
as required. Hence the proposition follows with $\eps = \delta^6$, as claimed.
\end{proof}

It is straightforward to deduce Theorem~\ref{thm:cycle:containers:turan} from Proposition~\ref{prop:containers_for_graphs}.

\begin{proof}[Proof of Theorem~\ref{thm:cycle:containers:turan}]
We apply Proposition~\ref{prop:containers_for_graphs} repeatedly, each time refining the set of containers obtained at the previous step. More precisely, suppose that after $t$ steps we have constructed a family $\C_t$ such that
$$|\C_t| \, \le \, \exp\bigg( \frac{n^{1 + 1/\ell}}{\eps}  \sum_{i=1}^t \max \Big\{ k(i)^{-1/(\ell-1)} \log k(i), \, n^{-(\ell-1) / \ell(2\ell-1)} \log n \Big\} \bigg),$$
$e(G) \le k(t) n^{1+1/\ell}$ for every $G \in \C_t$, and every $C_{2\ell}$-free graph is a subgraph of some $G \in \C_t$, where 
$$k(i) = \max\big\{ (1 - \eps)^i n^{1 - 1/\ell}, k_0 \big\}$$ 
and $k_0$ and $\eps$ are the constants given by Proposition~\ref{prop:containers_for_graphs}. We will construct a family $\C_{t+1}$ by applying Proposition~\ref{prop:containers_for_graphs} to each graph $G \in \C_t$ with more than $k(t+1) n^{1+1/\ell}$ edges. Finally, we will show that the family $\G_\ell(k) := \C_m$ obtained after $m$ iterations of this process has the required properties, for some suitably chosen $m \in \N$. 

Set $\C_0 = \{K_n\}$, and observe it satisfies the conditions above. Now, given such a family $\C_t$, for each $G \in \C_t$ we will define a collection of containers $\C(G)$ as follows: if $e(G) \le k(t+1) n^{1+1/\ell}$ then set $\C(G) = \{G\}$; otherwise apply Proposition~\ref{prop:containers_for_graphs} to $G$. We obtain a collection~$\C(G)$ of at most
 $$\exp\bigg( \frac{n^{1 + 1/\ell}}{\eps} \max \Big\{ k(t+1)^{-1/(\ell-1)} \log k(t+1), \, n^{-(\ell-1) / \ell(2\ell-1)} \log n \Big\} \bigg)$$
 subgraphs of~$G$ such that:
 \begin{itemize}
  \item[$(a)$] Every $C_{2\ell}$-free subgraph of~$G$ is a subgraph of some $C \in \C(G)$, and\smallskip
  \item[$(b)$] $e(C) \le (1 - \eps) e(G) \le k(t+1) n^{1+1/\ell}$ for every $C \in \C(G)$.
 \end{itemize}
Now simply set $\C_{t+1} = \bigcup_{G \in \C_t} \C(G)$, and observe that $\C_{t+1}$ satisfies the required conditions.

Finally, let us show that if $k \le n^{(\ell-1)^2 / \ell(2\ell-1)} / (\log n)^{\ell-1}$ and $m$ is chosen to be minimal so that $k(m) \le k$, then
$$|\C_m| \, \le \, \exp\Big( O(1) \cdot k^{-1/(\ell-1)} n^{1 + 1/\ell} \log k \Big)$$
as required. To see this, note first that $m = O(\log n)$, and that
$$n^{- (\ell-1) / \ell(2\ell-1)} (\log n)^2 \, = \, O(1) \cdot k^{-1/(\ell-1)} \log k,$$
by our upper bound on $k$. Since $k(i)$ decreases exponentially in $i$, it follows that
$$\sum_{i=1}^m \max \Big\{ k(i)^{-1/(\ell-1)} \log k(i), \, n^{-(\ell-1) / \ell(2\ell-1)} \log n \Big\} \, = \, O(1) \cdot k^{-1/(\ell-1)} \log k,$$
as claimed, and so the theorem follows.
\end{proof}

As noted above, Theorem~\ref{thm:cycle:containers} follows immediately from Theorem~\ref{thm:cycle:containers:turan} by choosing $k$ to be a suitably large constant. Moreover, Theorem~\ref{thm:main} and Theorem~\ref{cor:fewwithfewedges} are immediate consequences of Theorem~\ref{thm:cycle:containers}.

\begin{proof}[Proof of Theorem~\ref{thm:main}]
Let~$\G$ be the collection given by Theorem~\ref{thm:cycle:containers}. Since every $C_{2\ell}$-free graph is a subgraph of some~$G \in \G$, it follows that the number of $C_{2\ell}$-free graphs on $n$ vertices is at most
$$\sum_{G \in \G} 2^{e(G)} \, \le \, 2^{\delta n^{1 + 1/\ell}} \cdot 2^{C n^{1+1/\ell}} \, = \, 2^{O(n^{1+1/\ell})},$$
as required.
\end{proof}

\begin{proof}[Proof of Theorem~\ref{cor:fewwithfewedges}]
Given $\eps > 0$, let~$\G$ be the collection of graphs given by Theorem~\ref{thm:cycle:containers}, applied with $\delta = \eps/2$. Now, for any function $m = m(n)$ with $m = o\big( n^{1 + 1/\ell} \big)$, the number of $C_{2\ell}$-free graphs with $n$ vertices and at most $m$ edges is at most
$$\sum_{G \in \G} \sum_{s = 0}^{m} \binom{e(G)}{s} \, \le \, n^{1 + 1/\ell} \cdot 2^{\delta n^{1 + 1/\ell}} \binom{C n^{1+1/\ell}}{m} \, \le \, 2^{\eps n^{1+1/\ell}}$$
if $n$ is sufficiently large. Since $\eps > 0$ was arbitrary, the claimed bound follows.
\end{proof}

Moreover, it is easy to deduce the following theorem, which is only slightly weaker than Theorem~\ref{thm:randomturan}, from Theorem~\ref{thm:cycle:containers:turan} and Markov's inequality. The lower bound on $p(n)$ is best possible up to the polylog-factor, see Remark~\ref{rmk:randomturan:weak}, below.

\begin{thm}\label{thm:randomturan:weak}
For every $\ell \ge 2$, and every function $p = p(n) \gg n^{-(\ell-1) / (2\ell-1)} (\log n)^{\ell+1}$, 
$$\ex \big( G(n,p), C_{2\ell} \big) \, \le \, p^{1/\ell} n^{1+1/\ell} \log n$$
with high probability as $n \to \infty$. 
\end{thm}

\begin{proof}
Choose $k$ so that $p = k^{-\ell/(\ell-1)} \log k$, and let~$\G_\ell(k)$ be the collection of graphs given by Theorem~\ref{thm:cycle:containers:turan}. Observe that, if there exists a $C_{2\ell}$-free subgraph of $G(n,p)$ with $m$ edges, then some graph in $\G_\ell(k)$ must contain at least $m$ edges of $G(n,p)$. By Theorem~\ref{thm:cycle:containers:turan}, the expected number of such graphs is at most
$$\exp\Big( C k^{-1/(\ell-1)} n^{1 + 1/\ell} \log k \Big) \cdot {k n^{1+1/\ell} \choose m} \cdot p^m \, \le \, \left( \frac{O(1) \cdot p k n^{1+1/\ell}}{m} \right)^m \, \to \, 0$$ 
as $n \to \infty$ if $k \le n^{(\ell-1)^2 / \ell(2\ell-1)} / (\log n)^{\ell-1}$ and
$$m \, \gg \, \max\Big\{ pk  n^{1+1/\ell} , \, k^{-1/(\ell-1)} n^{1 + 1/\ell} \log k \Big\}.$$
Since one can check that these inequalities hold if 
$$p \, \gg \, n^{-(\ell-1) / (2\ell-1)} (\log n)^{\ell+1} \qquad \text{and} \qquad m \, \ge \, p^{1/\ell} n^{1+1/\ell} \log n,$$ 
the result follows.
\end{proof}

\begin{rmk}\label{rmk:randomturan:weak}
Note that, since $\ex \big( G(n,p), C_{2\ell} \big)$ is increasing in $p$, the bound in Theorem~\ref{thm:randomturan:weak} implies that, for every $p(n) \le n^{-(\ell-1) / (2\ell-1)} (\log n)^{2\ell}$, we have
$$\ex \big( G(n,p), C_{2\ell} \big) \, \le \, n^{1+1/(2\ell-1)} (\log n)^3$$
with high probability as $n \to \infty$. This bound is sharp up to the polylog-factor, cf.~\eqref{eq:KKSthm}. 
\end{rmk}

To finish this section, let us prove Proposition~\ref{prop:conj:implies:thm}. In fact, since the lower bound on $\ex(n,H)$ implicit in that statement is irrelevant to our counting argument (and moreover is usually unknown), we shall in fact prove the following rephrased version of the proposition.

\begin{defn}\label{def:ESgood}
Let us say that a bipartite graph $H$ is \emph{Erd\H{o}s-Simonovits good for a function $m = m(n)$} if there exist constants $C > 0$, $\eps > 0$ and $k_0 \in \N$ such that the following holds. Let $k \ge k_0$, and suppose that~$G$ is a graph with~$n$ vertices and~$k \cdot m(n)$ edges. Then there exists a (non-empty) collection~$\HH$ of copies of $H$ in $G$, satisfying
$$d_\HH(\sigma) \le \ds\frac{C \cdot |\HH|}{k^{(1 + \eps)(|\sigma| - 1)} e(G)} \quad \textup{for every $\sigma \subset V(\HH)$ with $1 \le |\sigma| \le e(H)$.}$$
\end{defn}

In this language, Conjecture~\ref{conj:balancedES} states that every bipartite graph $H$ is Erd\H{o}s-Simonovits good for the function $\ex(n,H)$.

\begin{prop}\label{prop:generalH}
Let $H$ be a bipartite graph and let $m \colon \N \to \N$ be a function. If $H$ is Erd\H{o}s-Simonovits good for $m$, then there are at most $2^{O(m(n))}$ $H$-free graphs on $n$ vertices. 
\end{prop}

\begin{proof}[Sketch proof]
The proof is almost identical to (and actually slightly simpler than) that of Theorem~\ref{thm:main}, so let us emphasize only the differences in the general case. We first prove a statement analogous to Proposition~\ref{prop:containers_for_graphs}, except that the collection $\C$ consists of at most
$$\exp\Big( k^{-\alpha} n^{1+1/\ell} \Big)$$
subgraphs of~$G$ which cover the $H$-free graphs on $n$ vertices, where $\alpha = \eps^2$, say. To do so, set $1/\tau = \delta^2 k^{1+\eps}$ and observe that, if $\delta < 1/C^2$, then
\begin{align*}
\delta(\HH,\tau) & \, = \, \frac{1}{e(\HH)} \,\sum_{j=2}^{e(H)} \,\frac{1}{\tau^{j-1}} \sum_{v \in V(\HH)} d^{(j)}(v) \\
& \, \le \, \frac{1}{e(\HH)} \sum_{j=2}^{e(H)}  \delta^{2(j-1)} k^{(1+\eps)(j-1)} \sum_{e \in E(G)} \frac{C \cdot e(\HH)}{k^{(1 + \eps)(j - 1)} e(G)} \, \le \, \delta.
\end{align*}
where $\HH$ denotes the $e(H)$-uniform hypergraph that encodes copies of $H$ in $G$.

The rest of the proof is exactly the same as above, except that our family $\G_\ell(k)$ of containers (as in Theorem~\ref{thm:cycle:containers:turan}) might consist of as many as $\exp\big( k^{-\alpha/2} n^{1 + 1/\ell} \big)$ graphs, each with at most $k n^{1 + 1/\ell}$ edges. We leave the details to the reader.
\end{proof}

\section{The Tur\'an problem on the Erd\H{o}s-R\'enyi random graph}\label{sec:Turan}

In this section we will show how to use the hypergraph container method in a slightly more complicated way in order to remove the unwanted factor of $\log n$ from the bound in Theorem~\ref{thm:randomturan:weak}, and hence to deduce Theorem~\ref{thm:randomturan}. 

Let us introduce some notation to simplify the statements which follow. First, let $\I = \I(n)$ denote the collection of $C_{2\ell}$-free graphs with $n$ vertices, and let $\G = \G(n,k)$ denote the collection of all graphs with $n$ vertices and at most $k n^{1+1/\ell}$ edges. By a \emph{coloured graph}, we mean a graph together with an arbitrary labelled partition of its edge set. 

The following structural result turns out to be exactly what we need.

\begin{thm}\label{thm:cycle:containers:turan:refined}
For each $\ell \ge 2$, there exists a constant $C = C(\ell)$ such that the following holds for all sufficiently large $n,k \in \N$ with $k \le n^{(\ell-1)^2 / \ell(2\ell-1)} / (\log n)^{2\ell-2}$. There exists a collection $\S$ of coloured graphs with $n$ vertices and at most $C k^{-1/(\ell-1)} n^{1 + 1/\ell}$ edges, and functions 
$$g \colon \I \to \S \qquad \text{and} \qquad h \colon \S \to \G(n,k)$$
with the following properties:
\begin{itemize}
\item[$(a)$] 
For every $s \ge 0$, the number of coloured graphs in $\S$ with $s$ edges is at most
$$\bigg( \frac{C n^{1+1/\ell}}{s} \bigg)^{\ell s}  \cdot \exp\Big( C k^{-1/(\ell-1)} n^{1 + 1/\ell} \Big).$$
\item[$(b)$] $g(I) \subset I \subset g(I) \cup h(g(I))$ for every $I \in \I$.
\end{itemize}
\end{thm}

Note that Theorem~\ref{thm:cycle:containers:turan:refined} implies Theorem~\ref{thm:cycle:containers:turan}. In order to prove Theorem~\ref{thm:cycle:containers:turan:refined}, we will need the following slight improvement of Theorem~\ref{thm:coveroff}, which was also proved by Balogh, Morris and Samotij~\cite[Proposition~3.1]{BMS} and by Saxton and Thomason~\cite[Theorem~6.2]{ST}.\footnote{To be precise, Theorem~6.2 in~\cite{ST} is stated where~$T$ is a tuple of vertex sets rather than a single vertex set, but it is straightforward to deduce this form from the methods of~\cite{ST}.}

\begin{thm}\label{thm:containers:turan}
Let $r \ge 2$ and let $0 < \delta < \delta_0(r)$ be sufficiently small. Let $\HH$ be an $r$-graph with~$N$ vertices, and suppose that $\delta(\HH,\tau) \le \delta$ for some $\tau > 0$. Then there exists a collection $\C$ of subsets of $V(\HH)$, and a function $f \colon V(\HH)^{(\le \tau N / \delta)} \to \C$ such that:
 \begin{itemize}
 \item[$(a)$] for every independent set $I$ there exists $T \subset I$ with $|T| \le \tau N / \delta$ and $I \subset f(T)$,\smallskip
 \item[$(b)$] $e\big( \HH[C] \big) \le \big(1 - \delta \big) e(\HH)$ for every $C \in \C$.
 \end{itemize}
\end{thm}

Note that Theorem~\ref{thm:containers:turan} implies Theorem~\ref{thm:coveroff}. The next step in the proof of Theorem~\ref{thm:cycle:containers:turan:refined} is the following strengthened version of Proposition~\ref{prop:containers_for_graphs}.

\begin{prop} \label{prop:refined_containers_for_graph}
For every $\ell \ge 2$, there exists $k_0 \in \N$ and $\eps > 0$ such that the following holds for every $k \ge k_0$ and every $n \in \N$. Set
\begin{equation}\label{def:mu}
\mu \, = \, \frac{1}{\eps} \cdot \max\Big\{ k^{-1/(\ell-1)}, \, n^{-(\ell-1) / \ell(2\ell-1)} \Big\}.
\end{equation}
Given a graph~$G$ with~$n$ vertices and $k n^{1+1/\ell}$ edges, there exists a function $f_G$ that maps subgraphs of $G$ to subgraphs of $G$, such that, for every $C_{2\ell}$-free subgraph~$I \subset G$, \begin{itemize}
  \item[$(a)$] There exists a subgraph $T=T(I) \subset I$ with $e(T) \le \mu n^{1+1/\ell}$ and $I \subset f_G(T)$, and\smallskip
  \item[$(b)$] $e\big( f_G(T(I)) \big) \le (1 - \eps) e(G)$.
 \end{itemize}
\end{prop}

The deduction of Proposition~\ref{prop:refined_containers_for_graph} from Theorem~\ref{thm:containers:turan} is identical to that of Proposition~\ref{prop:containers_for_graphs}  from Theorem~\ref{thm:coveroff}, and so we leave the details to the reader. 

In the proof of Theorem~\ref{thm:cycle:containers:turan:refined}, we will need the following straightforward lemma (see, for example,~\cite[Lemma~4.3]{CM}).

\begin{lemma}\label{obs:optimize}
Let $M > 0$, $s > 0$ and $0 < \delta < 1$. If $a_1,\ldots,a_m \in \RR$ satisfy $s = \sum_j a_j$ and $1 \le a_j \le (1-\delta)^j M$ for each $j \in [m]$, then
$$s \log s \, \le \, \sum_{j=1}^m a_j \log a_j + O(M).$$
\end{lemma} 

We can now deduce Theorem~\ref{thm:cycle:containers:turan:refined}. 

\begin{proof}[Proof of Theorem~\ref{thm:cycle:containers:turan:refined}]
We construct the functions $g$ and $h$ and the family $\S$ as follows. Given a $C_{2\ell}$-free graph $I \in \I$, we repeatedly apply Proposition~\ref{prop:refined_containers_for_graph}, first to the complete graph $G_0 = K_n$, then to the graph $G_1 = f_{G_0}(T_1) \setminus T_1$, where $T_1 \subset I$ is the set guaranteed to exist by part~$(a)$, then to the graph $G_2 = f_{G_1}(T_2) \setminus T_2$, where $T_2 \subset I \cap G_1 = I \setminus T_1$, and so on. We continue until we arrive at a graph $G_m$ with at most $k n^{1+1/\ell}$ edges, and set 
$$g(I) = (T_1, \ldots, T_m) \qquad \text{and} \qquad h\big( g(I) \big) = G_m.$$ 
Since $G_m$ depends only on the sequence $(T_1,\ldots,T_m)$, the function $h$ is well-defined.

It remains to bound the number of coloured graphs in $\S$ with $s$ edges. To do so, it suffices to count the number of choices for the sequence of graphs $(T_1,\ldots,T_m)$ with $\sum_j e(T_j) = s$. For each $j \ge 1$, define $k(j)$ and $\mu(j)$ as follows:
$$e\big( G_{m-j} \big) = k(j) n^{1+1/\ell} \quad \text{and} \quad \mu(j) = \frac{1}{\eps} \cdot \max\Big\{ k(j)^{-1/(\ell-1)}, \, n^{-(\ell-1) / \ell(2\ell-1)} \Big\},$$
and note that
$$k(j) \ge (1 - \eps)^{-j+1} k, \qquad T_{j+1} \subset G_j \qquad \text{and} \qquad e(T_{m-j}) \le \mu(j) n^{1+1/\ell}.$$
Thus, fixing $k$, $\eps$ and $s$ as above, and writing
$$\K(m) \, = \, \Big\{ \k = (k(1),\ldots,k(m)) \, : \, (1 - \eps)^{-j+1} k \le k(j) \le n^{1 - 1/\ell} \Big\}$$
for each $m \in \N$, and
$$\A(\k) \, = \, \Big\{ \a = (a(1),\ldots,a(m)) \, : \, a(j) \le \mu(j) n^{1 + 1/\ell} \text{ and } \sum_j a(j) = s \Big\},$$
for each $\k \in \K(m)$, it follows that the number of coloured graphs in $\S$ with $s$ edges is at most
$$\sum_{m = 1}^\infty \sum_{\k \in \K(m)} \sum_{\a \in \A(\k)} \prod_{j=1}^m { k(j) n^{1+1/\ell} \choose a(j)}.$$
Given $m \in \N$, $\k \in \K(m)$ and $\a \in \A(\k)$, let us partition the product over $j$ according to whether or not $\mu(j) = \frac{1}{\eps} \cdot n^{-(\ell-1) / \ell(2\ell-1)}$. Since $\K(m) = \emptyset$ for all $m \gg \log n$, the product of the terms for which this is the case is at most
$$\big( n^2 \big)^{\sum_j a(j)} \, \le \, \exp\Big( O(1) \cdot n^{1 + 1/\ell - (\ell-1) / \ell(2\ell-1)} (\log n)^2 \Big) \, \le \, \exp\Big( O(1) \cdot k^{-1/(\ell-1)} n^{1 + 1/\ell} \Big),$$
where in the last step we used the fact that $k \le n^{(\ell-1)^2 / \ell(2\ell-1)} / (\log n)^{2\ell-2}$. On the other hand, if $a(j) \le \frac{1}{\eps} \cdot k(j)^{-1/(\ell - 1)} n^{1 + 1/\ell}$, then
$${ k(j) n^{1+1/\ell} \choose a(j)} \, \le \, \bigg( \frac{e k(j) n^{1+1/\ell}}{a(j)} \bigg)^{a(j)} \, \le \, \bigg( \frac{n^{1+1/\ell}}{\eps a(j)} \bigg)^{\ell a(j)},$$
and hence, by Lemma~\ref{obs:optimize}, the product over the remaining $j$ is at most 
$$\bigg( \frac{C n^{1+1/\ell}}{s} \bigg)^{\ell s} \cdot \exp\Big( C k^{-1/(\ell-1)} n^{1 + 1/\ell} \Big)$$
for some $C = C(\ell)$. Noting that $\sum_{m = 1}^\infty \sum_{\k \in \K(m)} |\A(\k)| = n^{O(\log n)}$, the theorem follows.
\end{proof}

We can now easily deduce Theorem~\ref{thm:randomturan}.

\begin{proof}[Proof of Theorem~\ref{thm:randomturan}] 
Let $\ell \ge 2$ and note that, since $\ex( G, C_{2\ell} )$ is an increasing function of $E(G)$, it suffices to prove the claimed bound in the case $p \ge n^{-(\ell - 1) / (2\ell-1)} (\log n)^{2\ell}$. Given such a function $p = p(n)$, define $k = p^{-(\ell-1)/\ell}$ and (noting that $k \le n^{(\ell-1)^2 / \ell(2\ell-1)} / (\log n)^{2\ell-2}$) let $g$ and $h$ be the functions given by Theorem~\ref{thm:cycle:containers:turan:refined}. Suppose that there exists a $C_{2\ell}$-free subgraph $I \subset G(n,p)$ with $m$ edges, and observe that $g(I) \subset G(n,p)$, and that $G(n,p)$ contains at least $m - e\big( g(I) \big)$ elements of $h(g(I))$. The probability of this event is therefore at most 
\begin{align*}
\sum_{S \in \S} {k n^{1+1/\ell} \choose m - e(S)} p^m & \, \le \sum_{s = 0}^{C k^{-1/(\ell-1)} n^{1 + 1/\ell}} \bigg( \frac{C p^{1/\ell} n^{1+1/\ell}}{s} \bigg)^{\ell s} \exp\Big( C k^{-1/(\ell-1)} n^{1 + 1/\ell} \Big) \bigg( \frac{3pk n^{1+1/\ell}}{m - s} \bigg)^{m - s} \nonumber\\
& \, \le \, \exp\bigg[ O(1) \cdot \Big( p^{1/\ell} n^{1+1/\ell} + k^{-1/(\ell-1)} n^{1 + 1/\ell} \Big) \bigg] \bigg( \frac{4pk n^{1+1/\ell}}{m} \bigg)^{m/2} \to 0\label{eq:finalcounttozero}
\end{align*}
as $n \to \infty$, as long as
$$m \, \gg \, \max\Big\{ pk  n^{1+1/\ell} , \, k^{-1/(\ell-1)} n^{1 + 1/\ell} \Big\}.$$
Since one can check that these inequalities hold if 
$$p \, \ge \, n^{-(\ell-1) / (2\ell-1)} (\log n)^{2\ell} \qquad \text{and} \qquad m \, \gg \, p^{1/\ell} n^{1+1/\ell},$$ 
the theorem follows.
\end{proof}

\section{Complete bipartite graphs}\label{sec:Kst}

In this section we will prove Conjecture~\ref{conj:balancedES} for the complete bipartite graph $H = K_{s,t}$, under the  assumption that $\ex(n, K_{s,t}) = \Omega(n^{2-1/s})$, which is known to be the case when $t = t(s)$ is sufficiently large (see~\cite{ARS,BBK,Bukh,KRS}). The bound is generally believed to hold for every $2 \le s \le t$, and was conjectured already in 1954 by K\"ov\'ari, S\'os and Tur\'an~\cite{KST}. 

\begin{thm}\label{thm:Kst:ESgood}
For every $2 \le s \le t$, the graph $K_{s,t}$ is Erd\H{o}s-Simonovits good for $n^{2-1/s}$.
\end{thm}

Combining Theorem~\ref{thm:Kst:ESgood} with Proposition~\ref{prop:generalH}, we obtain a second proof\footnote{The proof of Corollary~\ref{cor:Kst} by Balogh and Samotij~\cite{BSmm,BSst} played an important role in the development of the hypergraph container method in~\cite{BMS}, so it is perhaps unsurprising that the method of this paper can be applied to $K_{s,t}$-free graphs. Nevertheless, the proof in~\cite{BSmm,BSst} is somewhat different to that presented here.} of the following breakthrough result of Balogh and Samotij~\cite{BSmm,BSst}.

\begin{cor}\label{cor:Kst}
For every $2 \le s \le t$, there are $2^{O(n^{2-1/s})}$ $K_{s,t}$-free graphs on $n$ vertices.
\end{cor}

Moreover, repeating the argument of Sections~\ref{sec:proof} and~\ref{sec:Turan}, we also obtain the following bounds for the Tur\'an problem on $G(n,p)$, which are likely to be close to best possible.

\begin{thm}\label{thm:Kst:Turan}
For every $2 \le s \le t$, there exists a constant $C  = C(s,t) > 0$ such that 
$$\ex \big( G(n,p), K_{s,t} \big) \, \le \,  \left\{
\begin{array} {c@{\quad}l} 
C n^{2 - (s+t-2) /(st-1)} (\log n)^2 & \textup{if } \; p \le n^{-(s - 1) / (st-1)} (\log n)^{2s/(s-1)} \\[+1ex]
C p^{(s-1)/s} n^{2-1/s} & \textup{otherwise}
\end{array}\right.$$
with high probability as $n \to \infty$. 
\end{thm}

Using the construction described in Section~\ref{sec:outline} (taking a blow-up and intersecting it with $G(n,p)$), one can easily show that, for each $2 \le s \le t$ such that $\ex(n, K_{s,t}) = \Omega(n^{2-1/s})$, we have $\ex \big( G(n,p), K_{s,t} \big) = \Omega\big( p^{(s-1)/s} n^{2-1/s}  \big)$ with high probability as $n \to \infty$. Moreover, another standard construction (remove one edge from each copy of $K_{s,t}$) shows that $\ex \big( G(n,p), K_{s,t} \big) = \Omega\big( n^{2 - (s+t-2) /(st-1)} \big)$ for every $p \ge n^{-(s+t-2)/(st-1)}$. We remark also that somewhat weaker upper bounds were obtained in~\cite{BSst}.

Let us fix $2 \le s \le t$. In order to prove Theorem~\ref{thm:Kst:ESgood} and Corollary~\ref{cor:Kst}, it is enough to prove a relatively weak balanced supersaturation theorem; however, to obtain the bounds in Theorem~\ref{thm:Kst:Turan} we need close to best possible bounds on $d_\HH(\sigma)$ for every $\sigma \subset E(G)$ with $\sigma \subset K_{s,t}$, where $G$ is a graph with $n$ vertices and $k n^{2 - 1/s}$ edges. In order to prove such a theorem, it will be useful to enrich the collection~$\HH$ (of Conjecture~\ref{conj:balancedES}) slightly, by recording the vertex partition of each copy of~$K_{s,t}$. In this section, therefore, a collection $\HH$ of copies of~$K_{s,t}$ in a graph~$G$ will be a collection of ordered pairs $(S,T)$ with $S \in V(G)^{(s)}$, $T \in V(G)^{(t)}$, and with $G[S,T]$ a complete bipartite graph.

When $\sigma \subset E(G)$ is an unlabelled set of edges, then $d_\HH(\sigma)$ retains the usual definition of the number of copies of~$K_{s,t}$ whose edge set contains~$\sigma$. However, in the proof below we will also need to define $d_\HH(A,B)$, where $A, B \subset V(G)$ with $1 \le |A| \le s$, $1 \le |B| \le t$ and $G[A,B]$ is a complete bipartite graph, to be the number of members of~$\HH$ for which the left vertex class contains~$A$ and the right vertex class contains~$B$, that is,
\[
 d_\HH(A,B) \, = \, \big| \big\{ (S,T) \in \HH : A \subset S \mbox{ and } B \subset T \big\} \big|.
\]
Moreover, for each $1 \le i \le s$ and $1 \le j \le t$, define
$$D^{(i,j)}(k,n) \, = \, \big( \delta k n^{(s-1)/s} \big)^{s-i} (\delta k^s)^{t-j},$$
for some sufficiently small constant $\delta > 0$. 

We will prove the following balanced supersaturation theorem.

\begin{thm}\label{thm:Kst:hypergraph}
For every $2 \le s \le t$, there exist constants $\delta > 0$ and $k_0 \in \N$ such that the following holds for every $k \ge k_0$ and every $n \in \N$. Given a graph~$G$ with~$n$ vertices and $k n^{2-1/s}$ edges, there exists a collection~$\HH$ of copies of~$K_{s,t}$ in~$G$, satisfying:
\begin{itemize}
 \item[$(a)$] $|\HH| = \Omega\big( k^{st} n^s \big)$, and\smallskip
 \item[$(b)$] $d_\HH(A,B) \le D^{(|A|,|B|)}(k,n)$ for every $A, B \subset V(G)$.
\end{itemize}
\end{thm}

Let us fix $2 \le s \le t$ and constants $\delta > 0$ (small) and $k_0 \in \N$ (large), and let $G$ be a graph $G$  with~$n$ vertices and $k n^{2-1/s}$ edges, where $k \ge k_0$. We will say that a pair $(A,B)$ of disjoint vertex sets with $1 \le |A| \le s$ and $1 \le |B| \le t$ is \emph{saturated} if 
$$d_\HH(A,B) \ge \lfloor D^{(|A|,|B|)}(k,n) \rfloor,$$ 
and that $(A,B)$ is \emph{good} if it contains no saturated pair $(A', B')$ with $A' \subset A$ and $B' \subset B$. We say that~$\HH$ is \emph{good} if every $(S,T) \in \HH$ is good. The main step in the proof of Theorem~\ref{thm:Kst:hypergraph} is the following proposition, cf.~Proposition~\ref{prop:finding:cycles}. 

\begin{prop}\label{prop:finding:kst}
Let $\HH$ be a good collection of copies of~$K_{s,t}$ in~$G$, with $e(\HH) = o\big( k^{st} n^s \big)$. There exists a copy $(S,T)$ of $K_{s,t}$, with $(S,T) \not\in \HH$, such that $\HH \cup \{(S,T)\}$ is good.
\end{prop}

In order to deduce Theorem~\ref{thm:Kst:hypergraph} from Proposition~\ref{prop:finding:kst}, we simply build up~$\HH$ edge by edge, until it has at least $\Omega( k^{st} n^s )$ edges. 

\begin{proof}[Sketch proof of Proposition~\ref{prop:finding:kst}]
Let~$\F$ to be the collection of saturated sets,
\[
 \F = \big\{ (A,B) \,:\, \emptyset \ne A, B \subset V(G) \mbox{ and } d_\HH(A, B) = \lfloor D^{(|A|, |B|)}(k,n)\rfloor \big\}.
\]
Provided that we do not pick~$S, T \subset V(G)$ with $A \subset S$ and $B \subset T$ for some $(A,B) \in \F$, then $\HH \cup \{(S,T)\}$ will be good.
By choosing a subgraph of~$G$ if necessary, we may assume that~$\F$ does not contain $(\{u\},\{v\})$ for any $\{u, v\} \in E(G)$ (observe that, since $|\HH| = o(k^{st} n^s)$, at most $o(e(G))$ of the edges of~$G$ correspond to saturated pairs, cf.~\eqref{eq:Delta1_bounded}).

For $A, B \subset V(G)$, define (cf.~the definition of $L_\F^{(1)}$ for cycles)
\begin{align*}
 X(A,B) & = \big\{ u \in V(G) : (A' \cup \{u\}, B') \in \F \mbox{ for some non-empty } A' \subset A, B' \subset B \big\}, \\
 Y(A,B) & = \big\{ v \in V(G) : (A', B' \cup \{v\}) \in \F \mbox{ for some non-empty } A' \subset A, B' \subset B \big\}.
\end{align*}
Following the proof of Lemma~\ref{lem:size_of_link}, it can easily be checked that 
\begin{equation}
 \label{eq:size_of_link}
 |X(A,B)| = O(\delta kn^{1-1/s}) \qquad\mbox{and}\qquad |Y(A,B)| = O(\delta k^s).
\end{equation}

Let $\M$ be the collection of all pairs $(S,v)$ with $v \in V(G)$, $S \in N_G(v)^{(s)}$, and such that $(S, \{v\})$ is good. We claim that
\begin{equation}
 \label{eq:size_of_M}
 |\M| = \Omega(k^s n^s).
\end{equation}
Indeed, for each $v \in V(G)$, observe that~$\M$ contains all pairs $(S,v)$ where~$S$ is generated as follows. Select an arbitrary $u_1 \in N_G(v)$. Now for $i=2, \ldots, s$, select
\[
 u_i \in N_G(v) \setminus \Big( \{u_1, \ldots, u_{i-1}\} \cup X\big(\{u_1, \ldots, u_{i-1}\}, \{v\}\big) \Big)
\]
and set $S = \{u_1, \ldots, u_s\}$. Since we choose $u_i \not\in X\big( \{u_1, \ldots, u_{i-1}\}, \{v\} \big)$, and using our assumption that~$\F$ does not contain any saturated pairs $(\{u_i\}, \{v\})$, it follows that the pair $(\{u_1, \ldots, u_i\}, \{v\})$ is good for every~$i \in [s]$, and hence $(S, \{v\})$ is also good. By~\eqref{eq:size_of_link}, the number of choices for each~$u_i$ is at least
\[
 |N_G(v)| - \big( s + O(\delta kn^{1-1/s}) \big).
\]
Thus for~$v$ whose degree is comparable with the average degree of~$G$, that is, $d_G(v) = \Omega(kn^{1-1/s})$, we have that the total number of choices for~$S$ is $\Omega(d_G(v)^s)$. Summing over all $v \in V(G)$ and using convexity, one obtains the bound~\eqref{eq:size_of_M}.

We now claim that there are $\Omega(k^{st}n^s)$ good pairs $(S,T)$ with $G[S,T] = K_{s,t}$. From this, the proposition follows immediately, since at least one of these is not in~$\HH$. Set
$$\M(S) = \big\{ v \in V(G) : (S,v) \in \M \big\}$$
for each $S \in V(G)^{(s)}$, and consider $S \in V(G)^{(s)}$ with $|\M(S)| = \Omega(k^s)$. We claim that there are $\Omega\big(|\M(S)|^t\big)$ sets $T \in \M(S)^{(t)}$ such that $(S,T)$ is good (and, since $T \subset \M(S)$, $G[S,T]$ is a complete bipartite graph). Indeed, for $i=1, \ldots, t$, we can pick an arbitrary vertex
\[
 v_i \in \M(S) \setminus \Big( \{v_1, \ldots, v_{i-1}\} \cup Y\big(S, \{v_1, \ldots, v_{i-1}\} \big) \Big),
\]
and set $T = \{v_1, \ldots, v_t\}$.
Since we chose $v_i \in \M(S)$ and $v_i \not\in Y\big(S, \{v_1, \ldots, v_{i-1}\} \big)$, it follows that $(S, \{v_1, \ldots, v_i\})$ is good for every~$i$ and hence $(S,T)$ is good.
By~\eqref{eq:size_of_link}, the number of choices for each~$v_i$ is at least
\[
 |\M(S)| - \big( t + O(\delta k^s) \big) = \Omega( |\M(S)|)
\]
Thus the total number of choices is $\Omega\big(|\M(S)|^t\big)$ as claimed.

Finally, observe that $\sum_{S \in V(G)^{(s)}} |\M(S)| = |\M|$, and thus for a typical $s$-set~$S$ we have $|\M(S)| \sim \binom{n}{s}^{-1} |\M| = \Omega(k^s)$. Summing over all $S \in V(G)^{(s)}$ with $|\M(S)| = \Omega(k^s)$ and using convexity, it follows easily that the total number of $K_{s,t}$ in~$G$ that do not contain a saturated set of vertices is at least $\Omega(k^{st} n^s)$. Thus there exists a good $K_{s,t}$ not already in~$\HH$, as required.
\end{proof}

Theorems~\ref{thm:Kst:ESgood} and~\ref{thm:Kst:Turan} and Corollary~\ref{cor:Kst} follow easily from Theorem~\ref{thm:Kst:hypergraph} using the method of Sections~\ref{sec:proof} and~\ref{sec:Turan}, and so we shall give only a very rough outline of the proof. Observe first that for every $\sigma \subset E(G)$, we have
$$d_\HH(\sigma) \, \le \, \max_{ij \,\ge\, |\sigma|} D^{(i,j)}(k,n),$$
and therefore the value of $\tau$ we require in order to apply Theorems~\ref{thm:coveroff} and~\ref{thm:containers:turan} is roughly
$$\max_{2 \,\le\, a \,\le\, st} \bigg( \frac{e(G)}{|\HH|} \max_{ij \,\ge\, a} \big( k n^{(s-1)/s} \big)^{s-i} k^{s(t-j)} \bigg)^{1/(a-1)} \approx\, \max\Big\{ k^{-s}, \, k^{-1} n^{-(s-1)^2/s(st-1)} \Big\},$$
where the approximation (which indicates equality up to a constant factor) follows from a short calculation\footnote{More precisely, show that the maximum occurs when $ij = a$ and $j = t$ (since $k \le n^{1/s}$), and note that the function $\big( kn^{2-1/s} \cdot k^{-st} n^{-s} \cdot \big( k n^{(s-1)/s} \big)^{s-i} \big)^{1/(ij-1)}$ is increasing in $i$ if and only if $k \ge n^{(s-1)/s(st-1)}$.}. We thus obtain the following analogue of Theorem~\ref{thm:cycle:containers:turan}. 

\begin{thm}\label{thm:Kst:containers:turan}
For each $2 \le s \le t$, there exists a constant $C = C(s,t)$ such that the following holds for all sufficiently large $n,k \in \N$ with $k \le n^{(s-1) / s(st-1)} / (\log n)^{1/(s-1)}$. There exists a collection~$\G_{s,t}(n,k)$ of at most
$$\exp\Big( C k^{-(s-1)} n^{2-1/s} \log k \Big)$$
graphs on vertex set~$[n]$ such that
$$e(G) \le k n^{2 - 1/s}$$ 
for every $G \in \G_{s,t}(n,k)$, and every $K_{s,t}$-free graph is a subgraph of some~$G \in \G_{s,t}(n,k)$. 
\end{thm}

Corollary~\ref{cor:Kst} follows easily from Theorem~\ref{thm:Kst:containers:turan}. To prove Theorem~\ref{thm:Kst:Turan}, we use a similar analogue of Theorem~\ref{thm:cycle:containers:turan:refined}, and repeat the argument of Section~\ref{sec:Turan}.  

\section*{Acknowledgements}

The authors would like to thank J\'ozsef Balogh, Wojciech Samotij and Andrew Thomason for many interesting discussions on independent sets in hypergraphs over the past few years. As the reader might imagine, these have had a significant bearing on the present work.

They would also like to thank Wojciech Samotij, Mikl\'os Simonovits, Xiaochuan Liu and Dhruv Mubayi for helpful comments on the manuscript, and David Conlon for interesting discussions on $C_{2\ell}$-free graphs.


\begin{thebibliography}{99} 

\bibitem{ARS} N.~Alon, L.~R\'onyai and T. Szab\'o, Norm-graphs: variations and applications, \emph{J. Combin. Theory, Ser. B}, \textbf{76} (1999), 280--290.

\bibitem{BBS1} J.~Balogh, B.~Bollob\'as and M.~Simonovits, On the number of graphs without forbidden subgraph, {\em J. Combin. Theory, Ser. B}, \textbf{91} (2004), 1--24.

\bibitem{BBS2} J.~Balogh, B.~Bollob\'as and M.~Simonovits, The typical structure of graphs without given excluded subgraphs, \emph{Random Structures Algorithms}, \textbf{34} (2009), 305--318.

\bibitem{BBS3} J.~Balogh, B.~Bollob\'as and M.~Simonovits, The fine structure of octahedron-free graphs, \emph{J. Combin. Theory, Ser. B}, \textbf{101} (2011), 67--84.

\bibitem{BMS} J.~Balogh, R.~Morris and W.~Samotij, Independent sets in hypergraphs, \emph{J. Amer. Math. Soc.}, \textbf{28} (2015), 669--709.

\bibitem{BSC4} J.~Balogh and W.~Samotij, Almost all $C_4$-free graphs have fewer than $(1 - \eps) \ex(n,C_4)$ edges, \emph{SIAM J. Discrete Math.}, \textbf{24} (2010), 1011--1018.

\bibitem{BSmm} J.~Balogh and W.~Samotij, The number of $K_{m,m}$-free graphs, \emph{Combinatorica}, \textbf{31} (2011), 131--150.

\bibitem{BSst} J.~Balogh and W.~Samotij, The number of $K_{s,t}$-free graphs, \emph{J. London Math. Soc.}, \textbf{83} (2011), 368--388.

\bibitem{Benson} C.~Benson, Minimal regular graphs of girths eight and twelve, \emph{Canad. J. Math.}, \textbf{18} (1966), 1091--1094.

\bibitem{BBK} P.V.M.~Blagojevi\'c, B.~Bukh and R.~Karasev, Tur\'an numbers for $K_{s,t}$-free graphs: topological obstructions and algebraic constructions, \emph{Israel J. Math.}, \textbf{197} (2013), 199--214.

\bibitem{MGT} B.~Bollob\'as, Modern Graph Theory (2nd edition), Springer, 2002.

\bibitem{BS74} J.A.~Bondy and M.~Simonovits, Cycles of even length in graphs, \emph{J. Combin. Theory, Ser. B}, \textbf{16} (1974), 97--105.

\bibitem{Brown} W.G.~Brown, On graphs that do not contain a Thomsen graph, \emph{Canad. Math. Bull.}, \textbf{9} (1966), 281--285.

\bibitem{Bukh} B.~Bukh, Random algebraic construction of extremal graphs, to appear in \emph{Bull. London Math. Soc.}

\bibitem{BZ} B. Bukh and Z. Jiang, A bound on the number of edges in graphs without an even cycle, to appear in \emph{Combin. Probab. Computing}.

\bibitem{CM} M.~Collares Neto and R.~Morris, Maximum-size antichains in random sets, to appear in \emph{Random Structures Algorithms}.

\bibitem{CG} D.~Conlon and W.T.~Gowers, Combinatorial theorems in sparse random sets, submitted.

\bibitem{DKLRS1} D. Dellamonica, Y. Kohayakawa, S. Lee, V. R\"odl and W. Samotij, On the number of $B_h$-sets, to appear in \emph{Combin. Prob. Computing}. 

\bibitem{DKLRS2} D. Dellamonica, Y. Kohayakawa, S. Lee, V. R\"odl and W. Samotij, The number of $B_3$-sets of a given cardinality, to appear in \emph{J. Combin. Theory, Ser. A}.

\bibitem{DKLRS3} D. Dellamonica, Y. Kohayakawa, S. Lee, V. R\"odl and W. Samotij, The number of $B_h$-sets of a given cardinality, submitted.


\bibitem{E38} P.~Erd\H{o}s, On sequences of integers no one of which divides the product of two others and on some related problems, \emph{Mitt. Forsch.-Inst. Math. Mech. Univ. Tomsk}, \textbf{2} (1938), 74--82.

\bibitem{E64} P.~Erd\H{o}s, Extremal problems in graph theory, M Fiedler (Ed.), Theory of Graphs and Its Applications, Proc. Symp. Smolenice, 1963 (2nd ed.), Academic Press, New York (1965), pp. 29--36.

\bibitem{EKR} P.~Erd\H{o}s, D.J.~Kleitman and B.L.~Rothschild, Asymptotic enumeration of $K_n$-free graphs, in {\em  Colloquio Internazionale sulle Teorie Combinatorie} (Rome, 1973), Vol. II, pp. 19--27. {\em Atti dei Convegni Lincei}, {\bf 17}, Accad. Naz. Lincei, Rome, 1976.

\bibitem{EFR} P.~Erd\H{o}s, P.~Frankl and V.~R\"odl, The asymptotic number of graphs not containing a fixed subgraph and a problem for hypergraphs having no exponent, \emph{Graphs and Combinatorics}, \textbf{2} (1986), 113--121.

\bibitem{ERS} P.~Erd\H{o}s, A.~R\'enyi and~V.T.~S\'os, On a problem of graph theory, \emph{Studia Sci. Math. Hungar.}, \textbf{1} (1966), 215--235.

\bibitem{ES66} P.~Erd\H{o}s and M.~Simonovits, A limit theorem in graph theory, \emph{Studia Sci. Math. Hungar.}, \textbf{1} (1966), 51--57.

\bibitem{ES82} P.~Erd\H{o}s and M.~Simonovits, Compactness results in extremal graph theory, \emph{Combinatorica}, \textbf{2} (1982), 275--288.

\bibitem{ES84} P.~Erd\H{o}s and M.~Simonovits, Cube-supersaturated graphs and related problems, Progress in graph theory (Waterloo, Ont., 1982), pp. 203--218, Academic Press, Toronto, ON, 1984.

\bibitem{ES46} P.~Erd\H{o}s and A.~Stone, On the structure of linear graphs, \emph{Bull. Amer. Math. Soc.}, \textbf{52} (1946), 1087--1091.

\bibitem{FS} R.~Faudree and M.~Simonovits, Cycle-supersaturated graphs, in preparation.

\bibitem{FuS} Z.~F\"uredi and M.~Simonovits, The history of degenerate (bipartite) extremal graph problems, Erd\H{o}s Centennial, Bolyai Soc Math Studies, {\bf 25} (eds L.~Lov\'asz, I.~Ruzsa and V.~T.-S\'os) pp.~169--264.

\bibitem{F96} Z.~F\"uredi, New asymptotics for bipartite Tura\'n numbers, \emph{J. Combin. Theory, Ser. A}, \textbf{75} (1996), 141--144.

\bibitem{FNV} Z.~F\"uredi, A.~Naor and J.~Verstra\"ete,  On the Tur\'an number for the hexagon, \emph{Adv. Math.}, \textbf{203} (2006), 476--496. 

\bibitem{HKL} P.E.~Haxell, Y.~Kohayakawa and T.~\L uczak, Tur\'an's Extremal Problem in Random Graphs: Forbidding Even Cycles, \emph{J. Combin. Theory, Ser. B}, \textbf{64} (1995), 273--287.

\bibitem{KW96} D.~Kleitman and D.~Wilson, On the number of graphs which lack small cycles, manuscript, 1996.

\bibitem{KW82} D.~Kleitman and K.~Winston, On the number of graphs without 4-cycles, \emph{Discrete Math.}, \textbf{41} (1982), 167--172.

\bibitem{KKS} Y. Kohayakawa, B. Kreuter and A. Steger, An extremal problem for random graphs and the number of graphs with large even-girth, \emph{Combinatorica}, \textbf{18} (1998), 101--120. 

\bibitem{KLRS} Y.~Kohayakawa, S.~Lee, V.~R\"odl and W.~Samotij, The number of Sidon sets and the maximum size of Sidon sets contained in a sparse random set of integers, \emph{Random Structures Algorithms}, \textbf{46} (2015), 1--25. 

\bibitem{KPR} Ph.G.~Kolaitis, H.J.~Pr\"omel and B.L.~Rothschild, $K_{\ell+1}$-free graphs: asymptotic structure and a $0$-$1$ law, \emph{Trans. Amer. Math. Soc.}, \textbf{303} (1987), 637--671.

\bibitem{KRS} J.~Koll\'ar, L.~R\'onyai and T.~Szab\'o, Norm-graphs and bipartite Tur\'an numbers, \emph{Combinatorica}, \textbf{16} (1996), 399--406.

\bibitem{KST} T.~K\H{o}v\'ari, V.~S\'os and P.~Tur\'an, On a problem of K.~Zarankiewicz, \emph{Colloq. Math.}, \textbf{3} (1954), 50--57.

\bibitem{Kreuter} B.~Kreuter, Extremale und Asymptotische Graphentheorie f\"ur verbotene bipartite Untergraphen, Diplomarbeit, Forschungsinstitut f\"ur Diskrete Mathematik, Universit\"at Bonn, 1994.

\bibitem{LUW} F.~Lazebnik, V.A.~Ustimenko and A.J.~Woldar, Properties of certain families of $2k$-cycle free graphs, \emph{J. Combin. Theory, Ser. B}, \textbf{60} (1994), 

\bibitem{NRS} B.~Nagle, V.~R\"odl and M.~Schacht, Extremal hypergraph problems and the regularity method, Topics in Discrete Mathematics, \emph{Algorithms Combin.}, \textbf{26} (2006), 247--278.

\bibitem{P12} O. Pikhurko, A note on the Tur\'an function of even cycles, \emph{Proc. Amer. Math Soc.}, \textbf{140} (2012), 3687--3992.

\bibitem{PS} H.J.~Pr\"omel and A.~Steger, The asymptotic number of graphs not containing a fixed color-critical subgraph, \emph{Combinatorica}, \textbf{12} (1992), 463--473.

\bibitem{RS} V. R\"odl and M.~Schacht, Extremal results in random graphs, to appear in the Erd\H{o}s Centennial volume of \emph{Bolyai Soc. Math. Stud.}

\bibitem{Sax} D.~Saxton, Independent sets, hereditary properties and point systems, PhD thesis, University of Cambridge, 2012.

\bibitem{ST} D.~Saxton and A.~Thomason, Hypergraph containers, \emph{Inventiones Math.}, \textbf{201} (2015), 1--68.

\bibitem{Schacht} M.~Schacht, Extremal results for random discrete structures, submitted.

\bibitem{Sim82} M. Simonovits, Extremal graph problems, degenerate extremal problems and super-saturated graphs, in Progress in graph theory (Waterloo, Ont., 1982), pp. 419--437, Academic Press, Toronto, ON, 1984.

\bibitem{T41} P. Tur\'an, Eine Extremalaufgabe aus der Graphentheorie, \emph{Mat. Fiz. Lapok}, \textbf{48} (1941), 436--452.

\bibitem{V00} J.~Verstra\"ete, On arithmetic progressions of cycle lengths in graphs, \emph{Combin. Probab. Computing}, \textbf{9} (2000), 369--373.

\bibitem{W91} R.~Wenger, Extremal graphs with no $C_4$, $C_6$, or $C_{10}$, \emph{J. Combin. Theory, Ser. B}, \textbf{52} (1991), 113--116.

\end{thebibliography}
\end{document}